\documentclass[10pt, reqno]{amsart}
\usepackage{amssymb}
\usepackage{amsmath,amscd}
\usepackage[initials,nobysame,alphabetic]{amsrefs}
\usepackage{amsmath, amssymb}
\usepackage{amsfonts}
\usepackage{mathrsfs}
\usepackage{mathpazo}
\usepackage[arrow,matrix,curve,cmtip,ps]{xy}
\usepackage[margin=1.3in]{geometry}
\usepackage{amsthm}
\usepackage{float}
\usepackage{hyperref}
\usepackage{graphics}
\usepackage{graphicx}
\usepackage{verbatim}
\usepackage{multirow}
\usepackage{tikz}
\usepackage{enumerate}
\allowdisplaybreaks
\newtheorem{theorem}{Theorem}[section]
\newtheorem{lemma}[theorem]{Lemma}
\newtheorem{proposition}[theorem]{Proposition}
\newtheorem{corollary}[theorem]{Corollary}

\newtheorem{definition}[theorem]{Definition}

\newtheorem{remark}[theorem]{Remark}

\DeclareMathOperator*{\esssup}{ess\,sup}
\DeclareMathOperator*{\essinf}{ess\,inf}
\numberwithin{equation}{section}

\newcommand{\E}{\mathbb{E}}
\newcommand{\R}{\mathbb{R}}

\def \nn {\nonumber}
\def \cS {\mathcal{S}}
\def \cH {\mathcal{H}}
\def \cP {\mathcal{P}}
\def \bP {\mathbb{P}}
\def\bF {\mathbb{ F}}
\def\cF {\mathcal{ F}}

\def\cD{\mathcal{ D}}
\def\cL{\mathcal{L}}
\def \gm {\gamma}
\def \bt{\beta}
\def \time{[0,T]}

\def \cT{\mathcal{T}}

\def \dis{\displaystyle}
\def \p {\mathbb{P}}
\def \t {\tau}
\def \h{\eta}
\def \th {\theta}
\def \nd{\noindent}
\def \ms{\medskip}
\def \bs{\bigskip}
\def \rw{\rightarrow}
\def\ft {\infty}
\def \lb{\label}

\def \nn {\nonumber}

\def \mx{\mbox}
\def \rwft {\rightarrow \infty}
\def \ed {\end{document}}
\def \og {\omega}
\def \s{\sigma}
\def \ca {\mathcal{A}}

\def \tt {t\le T}
\title[Mean-field Doubly Reflected backward SDEs and  stochastic differential games]{Mean-field Doubly Reflected backward stochastic differential equations}
\author{Yinggu Chen, Said Hamad\`ene and Tingshu Mu}
\address{Department of Mathematics, Shandong University, Jinan, Shandong Province, China}
\email{272564198@qq.com}
\address{Le Mans University, LMM \\ Avenue Olivier Messiaen \\ 72085 Le Mans, Cedex 9, France}
\email{ hamadene@univ-lemans.fr}
\address{Le Mans University, LMM \\ Avenue Olivier Messiaen \\ 72085 Le Mans, Cedex 9, France}
\email{ tingshu.mu.etu@univ-lemans.fr}
\begin{document}
\maketitle
\date{\today}
\begin{abstract} We study mean-field doubly reflected BSDEs. First, using the fixed point method, we show existence and uniqueness of the solution when the data which define the BSDE are $p$-integrable with $p=1$ or $p>1$. The two cases are treated separately. Next by penalization we show also the existence of the solution. The two methods do not cover the same set of assumptions.
\end{abstract}
\def \sp {\cS^p}
\def \udl{\underline}
\def \hlcd {\cH^{d}_{loc}}
\def \d{\delta}
\def \tl {\tilde}
\def \td {[T-\d,T]}
\def \tdd {[T-2\d,T-\d]}
\ms

\par
\section{Introduction}
In this paper we are concerned with the problem of existence and uniqueness of a solution of the doubly reflected BSDE of the following type: For any $t\le T$,
\begin{equation}\label{MFDRBSDE_intro}
\left\{
\begin{aligned}
&Y_t = \xi + \int_t^T f(s, Y_s, \E[Y_s],Z_s)ds +K^+_T-K^+_t - K^-_T +K^-_t -\int_t^T Z_s dB_s\,;\\
&h( t,\og,Y_t, \E[Y_t]) \leq Y_t \leq g( t,\og,Y_t, \E[Y_t]);\\
&\int_0^T(Y_s-h(s,Y_s, \E[Y_s]))dK^+_s=\int_0^T(Y_s-g(s, Y_s, \E[Y_s]))dK^-_s=0\,\, (K^\pm \mbox{ are increasing processes }).
\end{aligned}\right.
\end{equation}
It is said associated with the quadruple $(f,\xi,h,g)$. Those BSDEs are of mean-field type because the generator $f$ and the barriers $h$ and $g$ depend on the law of $Y_t$ through its expectation. For simplicity reasons we stick to this framework, however it can be somewhat generalized (see Remark \ref{gene}).

Since the introduction by Lasry and Lions \cite{ll1} of the general mathematical modeling
approach for high-dimensional systems of evolution equations corresponding to
a large number of "agents" (the mean-field model), the interest to the mean-field models grows steadily in connection with several applications. Later standard mean-field BSDEs have been introduced in \cite{bdlp09}. Since then, there have been several papers on mean-field BSDEs including (\cites{beh,blp09,blpr17,cd, mfrbsde2019,li14,pham19,pham16}, etc) in relation with several fields and motivations in mathematics and economics, such as stochastic control, games, mathematical finance,
utility of an agent inside an economy, PDEs, actuary, etc.

Mean-field one barrier reflected BSDEs have been considered first in the paper \cite{li14}. This latter generalizes the work in \cite{blp09} to the reflected framework. Later Briand et al. \cite{beh} have considered another type of one barrier mean-field reflected BSDEs. Actually in \cite{beh}, the reflection of the component $Y$ of the solution holds in expectation. They show existence and uniqueness of the solution when the increasing process, which makes the constraint on $Y$ satisfied, is deterministic. Otherwise the solution is not necessarily unique. The main motivation is the assessment of the risk of a position in a financial market.

In \cite{mfrbsde2019}, Djehiche et al. consider the above problem \eqref{MFDRBSDE_intro} when there is only one reflecting barrier (e.g. take $g\equiv +\infty$). The authors show existence and uniqueness of the solution in several contexts of integrability of the data $(f,\xi,h)$. The methods are the usual ones: Fixed point and penalization. Those methods do not allow for the same set of assumptions. For example, the fixed point method does not allow generators which depend on $z$ while the penalization does, at the price of some additional regularity properties which are not required by the use of the first method.
The main motivation for considering such a problem comes from actuary and namely the assessment of the prospective reserve of a representative contract in life-insurance (see e.g. \cite{mfrbsde2019} for more details). Later, there have been several works on this subject including (\cites{dddb,ddz,hmw}).

In this paper we consider the extension of the framework of \cite{mfrbsde2019} to the case of two reflecting barriers. We first show existence and uniqueness of a solution of \eqref{MFDRBSDE_intro} by the fixed point method. We deal with the case when the data of the problem are only integrable or $p$-integrable with $p>1$. Those cases are treated separately because one cannot deduce one of them from the other one.
The generator $f$ does not depend on $z$ while the main requirement on $h$ and $g$ is only to be Lipschitz continuous with small enough Lipschitz constants (see condition \eqref{cd1}).

Later on, we use the penalization method to show the existence of a solution for \eqref{MFDRBSDE_intro} under an adaptation to our framework, of the well-known  Mokobodski condition (see e.g. \cites{ck96,hl00,bhm}, etc.) which plays an important role. Within this method, $f$ may depend on $z$ while the Lipschitz property of $h$ and $g$ is replaced with a monotonicity one. As a by-product, we provide a procedure to approximate the solution of \eqref{MFDRBSDE_intro} by a sequence of solutions of standard mean-field BSDEs.

The paper is organized as follows: In Section 2, we fix the notations and the frameworks. In Section 3, we deal with the case when $p>1$ and finally with the case $p=1$. Section 4 is devoted to the study of the penalization scheme which we show it is convergent and its limit provides a solution for
\eqref{MFDRBSDE_intro}. The adapted
Mokobodski condition plays an important role since it makes that the approximations of the processes $K^\pm$ have mild increments and do not explode. As a by-product when the solution of \eqref{MFDRBSDE_intro} is unique, this scheme provides a way to approximate this solution.
\section{Notations and formulation of the problems}
\subsection{Notations}
Let $T$ be a fixed positive constant. Let $(\Omega,\mathcal{F},\mathbb{P})$
denote a complete probability space with $B=(B_t)_{t\in[0,T]}$ a \textit{d-}dimensional
Brownian motion whose natural filtration is $(\mathcal{F}^0_t:=\sigma\lbrace{B_s,s\leq t
\rbrace})_{0\leq t\leq T}$. We denote by $\mathbb{F}=(\mathcal{F}_{t})_{0\leq t\leq T}$
the completed filtration of $(\mathcal{F}^0_{t})_{0\leq t\leq T}$ with the $\mathbb{P}$-null
sets of $\mathcal{F}$, then it satisfies the usual conditions. On the other hand, let $\mathcal{P}$ be the $\sigma$-algebra on $[0,T]\times \Omega$ of the $\mathbb{F}$-progressively
measurable sets.\\

For $p\geq 1$ and $0\le s_0<t_0\le T$,  we define the following spaces:
\begin{itemize}
\item $L^p:=\lbrace \xi: \cF_T-\text{measurable radom variable s.t.}\; \E[|\xi|^p]<\infty\rbrace;$
\item $\cH^{p,d}_{loc}:=\lbrace (z_t)_{t\in\time}: \cP-\text{measurable process and}\;\R^d-\text{valued  s.t. }\E\left(\int_0^T|z_s(\omega)|^2ds\right)^{\frac{p}{2}} <\infty\}$.
\item $\cH^{m}_{loc}:=\lbrace (z_t)_{t\in\time}: \cP-\text{measurable process and}\;\R^m-\text{valued  s.t. }\p-a.s. \int_0^T|z_s(\omega)|^2ds <\infty\}$; $\bar z\in \cH^{m}_{loc}([s_0,t_0])$ if $\bar z_r=z_r1_{[s_0,t_0]}(r)$, $dr\otimes d\p$-a.e. with $z\in \cH^{m}_{loc}$.
\item $\cS^p_c:=\lbrace ({y}_t)_{t\in\time}: \;\text{continuous and
$\cP$-measurable process s.t.}\;
\E[\sup_{t\in\time}|y_t|^p]<\infty\rbrace$; $\cS^p_c([s_0,t_0])$ is the space $\cS^p_c$ reduced to the interval $[s_0,t_0]$. If $y\in \cS^p_c([s_0,t_0])$, we denote by
 $\Vert y\Vert_{\cS^p_c([s_0,t_0])}:=\{\E[\sup_{s_0\leq u\leq t_0}|{y}_u|^p]\}^{\frac{1}{p}}$.
\item $\ca:=\lbrace (k_t)_{t\in\time}:\;\text{continuous, $\cP$-measurable and non-decreasing process s.t.}\;  k_0=0 \rbrace$;
 $\ca([s_0,t_0])$ is the space $\ca$ reduced to the interval $[s_0,t_0]$ (with $k_{s_0}=0$);
\item $\cS^2_{ci}:=\lbrace ({y}_t)_{t\in\time}: \;\text{continuous and $\cP$-measurable and non-decreasing process s.t.}\; \E[\sup_{t\in\time}|y_t|^2]<\infty\rbrace$;
\item $\mathcal{T}_t:=\lbrace \tau, \mathbb{F}-\text{stopping time s.t.}\;\bP-a.s.\,\,\, t\le \tau\leq T\rbrace$;
\item $\cD:=\lbrace (\phi)_{t\in\time}: \bF-\text{adapted}, \;\R-\text{valued continuous process s.t.}\;\Vert\phi\Vert_1=\sup_{\tau\in\cT_0}\E[|y_\tau|]<\infty\rbrace$. Note that the normed space $(\cD,\Vert.\Vert_1)$ is complete (e.g. \cite{dlm}, pp.90). We denote by $(\cD([s_0,t_0]),\Vert.\Vert_1)$, the restriction of $\cD$ to the time interval $[s_0,t_0]$. It is a complete metric space when endowed with the norm $\Vert.\Vert_1$ on $[s_0,t_0]$, i.e.,
\[\Vert X\Vert_{1,[s_0,t_0]}:=\sup_{\tau\in\cT_0,\,\,s_0\leq\tau\leq t_0}\E[|X_\tau|]<\infty. \]
\end{itemize}
\subsection{The class of doubly reflected BSDEs }In this paper we aim at finding $\cP$-measurable processes \\$(Y, Z, K^+, K^-)$  which solves the doubly reflected BSDE of mean-field type associated with the generator $f(t,\omega,y,y')$, the terminal condition $\xi$, the lower barrier $h(y,y')$, and the upper barrier $g(y,y')$ when the data are $p$-integrable, in the cases $p>1$ and $p=1$ respectively. The two cases should be considered separately since one cannot deduce one case from another one. So to begin with let us make precise definitions: \begin{definition}
We say that the quaternary of $\cP$-measurable processes $(Y_t, Z_t, K^+_t, K^-_t)_{t\le T}$ is a solution of the mean-field reflected BSDE associated with $(f, \xi, h, g)$ if :
\ms

\nd \underline{Case}: $p>1$
\begin{equation}\label{MFDRBSDE P>1}
\left\{
\begin{aligned}
&Y \in \cS_c^p,\quad Z \in \hlcd \quad and\quad K^+, K^- \in \ca;\\
&Y_t = \xi + \int_t^T f(s, Y_s, \E[Y_s])ds +K^+_T-K^+_t - K^-_T +K^-_t -\int_t^T Z_s dB_s,\quad 0\leq t\leq T;\\
&h( t,Y_t, \E[Y_t]) \leq Y_t \leq g(t, Y_t, \E[Y_t]), \quad \forall t\in [0,T]; \\
&\int_0^T(Y_s-h(s,Y_s, \E[Y_s]))dK^+_s=\int_0^T(Y_s-g(s, Y_s, \E[Y_s]))dK^-_s=0.
\end{aligned}\right.
\end{equation}\par
\nd \udl{Case}: $p=1$,
\begin{equation}\label{MFDRBSDE P=1}
\left\{
\begin{aligned}
&Y \in \cD,\quad  Z \in\hlcd \quad and\quad K^+, K^- \in \ca;\\
&Y_t = \xi + \int_t^T f(s, Y_s, \E[Y_s])ds +K^+_T-K^+_t - K^-_T +K^-_t -\int_t^T Z_s dB_s,\quad 0\leq t\leq T;\\
&h( t,Y_t, \E[Y_t]) \leq Y_t \leq g(t, Y_t, \E[Y_t]), \quad \forall t\in [0,T]; \\
& \int_0^T(Y_s-h(s, Y_s, \E[Y_s]))dK^+_s=\int_0^T(Y_s-g( s,Y_s, \E[Y_s]))dK^-_s=0.\\
\end{aligned}\right.
\end{equation}
\end{definition}

\subsection{Assumptions on $(f, \xi, h, g)$}
We now make precise requirements on the items $(f, \xi, h, g)$ which define the doubly reflected backward stochastic differential equation of mean-field type.
\ms

\noindent \textbf{\underline{Assumption (A1)}}:$\ $\\\\
\noindent(i) The coefficients $f$, $h$, $g$ and $\xi$ satisfy:\par
\quad (a) $f$ does not depend on $z$ and the process $(f(t, 0, 0))_{t\leq T}$ is $\cP$- measurable and such that $\int_0^T|f(t, 0, 0)|dt \in L^p$; \par
\quad (b) $f$ is Lipschitz w.r.t $(y,y')$ uniformly in$(t, \omega)$, i.e., there exists a positive constant $C_f$ such that $\bP$- a.s. for all $t \in [0, T]$, $y_1,y'_1, y_2$ and $y'_2$ elements of $\R$,
\begin{equation}
|f(t,\omega,y_1,y'_1)-f(t,\omega,y_2,y'_2)| \leq C_f (|y_1-y'_1|+|y_2-y'_2|).
\end{equation}\par
\noindent(ii) $h$ and $g$ are  mappings from $[0,T]\times\Omega\times \R^2$ into $\R$ which satisfy: \par
\quad (a) $h$ and $g$ are Lipschitz w.r.t. $(y,y')$, i.e., there exist  pairs of positive constants $(\gm_1, \gm_2)$, $(\bt_1,\bt_2)$ such that for any $t,x, x', y$ and $y'\in\R,$
\begin{equation}\lb{lipgh}
\begin{aligned}
|h(t,\og,x,x')-h(t,\og,y,y')| \leq \gm_1 |x-y|+\gm_2|x'-y'|,\\
|g(t,\og,x,x')-g(t,\og,y,y')| \leq \bt_1 |x-y|+\bt_2|x'-y'|;
\end{aligned}
\end{equation}\par
\quad (b) $\p$-a.s., $h(t,\og,x,x')<g(t,\og,x,x'),$ for any $t, x,x'\in \R$;\par
\quad (c) the processes $(h(t,\og,0,0))_{t\le T}$ and $(g(t,\og,0,0))_{t\le T}$ belong to $\cS_c^p$
 (when $p>1$) and are continuous of class  $\cD$ (when $p=1$).\par
\noindent(iii) $\xi$ is an $\cF_T$- measurable, $\R$-valued r.v., $\E[|\xi|^p]<\infty$ and $\p$-a.s., $h(T,\xi, \E[\xi])\leq \xi \leq g(T,\xi, \E[\xi])$.
\section{Existence and Uniqueness of a Solution of the Doubly Reflected BSDE of Mean-Field type}
\par
Let $Y=(Y_t)_{t\le T}$ be an $\R$-valued, $\cP$-measurable process and $\Phi$ the mapping  that associates to $Y$ the following process  $(\Phi(Y)_t)_{t\le T}$: $\forall t\le T$,
$$
\begin{aligned}
\Phi(Y)_t :&=\esssup_{\tau\geq t} \essinf_{\sigma\geq t}\E\{\int_t^{\sigma\wedge\tau}f(s,Y_s,\E[Y_s])ds +g(\s,Y_\sigma, \E[Y_t]_{t=\sigma})\mathbb{1}_{\{\sigma<\tau\}}\\
&\qquad \qquad \qquad \qquad +h( \t,Y_\tau, \E[Y_t]_{t=\tau})\mathbb{1}_{\{\tau\leq\sigma, \t<T\}}+\xi\mathbb{1}_{\{\tau=\sigma=T\}}|\cF_t\}.
\end{aligned}
$$For the well-posedness of $\Phi(Y)$ one can see e.g. \cite{lepmaing}, Theorem 7.
\ms

The following result is related to some properties of $\Phi(Y)$.
\begin{lemma}\label{lemma31}Assume that assumptions (A1) are satisfied for $p=1$ and $Y\in \cD$. Then the process $\Phi(Y)$ belongs to $\cD$. Moreover there exist processes $(\underbar Z_t)_{t\le T}$ and
$(\underbar A^\pm_t)_{t\le T}$ such that:
\begin{equation}\label{eq:jeux}\left\{\begin{array}{l}
\underbar Z\in \cH^{m}_{loc}; \,\,\underbar A^\pm \in \ca;\\\\
\Phi(Y)_t=\xi+\int_t^Tf(s,Y_s,\E[Y_s])ds+\underbar A^+_T-\underbar A^+_t-\underbar A^-_T+\underbar A^-_t-\int_t^T\underbar Z_sdB_s,\,\,t\le T; \\\\
h(t,Y_t,\E[Y_t])\le \Phi(Y)_t\le g(t,Y_t,\E[Y_t]),\,\,t\le T;\\\\
\int_0^T(\Phi(Y)_t-h(t,Y_t,\E[Y_t]))d
\underbar A^+_t=\int_0^T(\Phi(Y)_t-g(t,Y_t,\E[Y_t]))d
\underbar A^-_t=0.\end{array}\right.
\end{equation}
\end{lemma}
\begin{proof}First note that since $Y\in \cD$ and $g,h$ are Lipschitz then the processes $(h(t,Y_t,\E[Y_t]))_{t\le T}$ and $(g(t,Y_t,\E[Y_t]))_{t\le T}$ belong also to $\cD$. Next as $h<g$ then, using a result by \cite{imen}, Theorem 4.1 or
\cite{topo}, Theorem 3.1, there exist $\cP$-measurable processes $(\underbar Y_t)_{t\le T}$,  $(\underbar Z_t)_{t\le T}$ and
$(\underbar A^\pm_t)_{t\le T}$ such that:
$$\left\{\begin{array}{l}
\underbar Y\in \cD;\,\,\underbar Z\in \cH^{m}_{loc}; \underbar A^\pm \in \ca;\\\\
\underbar Y_t=\xi+\int_t^Tf(s,Y_s,\E[Y_s])ds+\underbar A^+_T-\underbar A^+_t-\underbar A^-_T+\underbar A^-_t-\int_t^T\underbar Z_sdB_s,\,\,t\le T; \\\\
h(t,Y_t,\E[Y_t])\le \underbar Y_t\le g(t,Y_t,\E[Y_t]),\,\,t\le T;\\\\
\int_0^T(\underbar Y_t-h(t,Y_t,\E[Y_t]))d
\underbar A^+_t=\int_0^T(\underbar Y_t-g(t,Y_t,\E[Y_t]))d
\underbar A^-_t=0.\end{array}\right.
$$
Let us point out that in \cite{topo}, Theorem 3.1, the result is obtained in the discontinuous framework, namely the obstacles are right continuous with left limits processes. However since in our case the processes $(h(t,Y_t,\E[Y_t]))_{t\le T}$ and $(g(t,Y_t,\E[Y_t]))_{t\le T}$ are continuous then $\underbar Y$ and
$\underbar A^\pm$ are continuous, and the frameworks of \cite{imen} and \cite{topo} are the same (one can see e.g. \cite{lepxu}, pp.60). Finally the process $\underbar Y$ has the following representation as the value of a zero-sum Dynkin game: $\forall t\le T$,
\begin{equation}\label{dynkinlemme}
\begin{aligned}
\underbar Y_t :&=\esssup_{\tau\geq t} \essinf_{\sigma\geq t}\E\{\int_t^{\sigma\wedge\tau}f(s,Y_s,\E[Y_s])ds +g(\s, Y_\sigma, \E[Y_t]_{t=\sigma})\mathbb{1}_{\{\sigma<\tau\}}\\
&\qquad \qquad \qquad \qquad+h(\t, Y_\tau, \E[Y_t]_{t=\tau})\mathbb{1}_{\{\tau\leq\sigma, \t<T\}}+\xi\mathbb{1}_{\{\tau=\sigma=T\}}|\cF_t\}.
\end{aligned}
\end{equation}Therefore $\underbar Y=\Phi(Y)$ and the claim is proved.
\end{proof}
\begin{remark}
Note that we have also for any $t\le T$,
\begin{equation}\label{dynkinlemme2}
\begin{aligned}
\underbar Y_t :&= \essinf_{\sigma\geq t}\esssup_{\tau\geq t}\E\{\int_t^{\sigma\wedge\tau}f(s,Y_s,\E[Y_s])ds +g(\s, Y_\sigma, \E[Y_t]_{t=\sigma})\mathbb{1}_{\{\sigma<\tau\}}\\
&\qquad \qquad \qquad \qquad+h(\t, Y_\tau, \E[Y_t]_{t=\tau})\mathbb{1}_{\{\tau\leq\sigma, \t<T\}}+\xi\mathbb{1}_{\{\tau=\sigma=T\}}|\cF_t\}.
\end{aligned}
\end{equation}
\end{remark}
\subsection{The case $p>1$ }
We will first show that $\Phi$ is well defined from $\sp_c$ to $\sp_c$.
\begin{lemma}
Let$f, h, g$ and $\xi$ satisfy Assumption (A1) for some $p>1$. If $Y \in \cS^p_c$ then $\Phi(Y)  \in \cS^p_c$.
\end{lemma}
\begin{proof}  Let $Y\in \sp_c$. For $\s$ and $\tau$ two stopping times, let us define:  $$
\begin{aligned}
\cL(\t,\s) &=\int_0^{\t\wedge \s}f(r,Y_r,\E[Y_r])dr +g(\s,Y_\s, \E[Y_t]_{t=\s})\mathbb{1}_{\{\s<\t\}}+h(\t,Y_\t, \E[Y_t]_{t=\t})\mathbb{1}_{\{\t\leq \s, \t<T\}}+\xi\mathbb{1}_{\{\t=\s=T\}}.
\end{aligned}
$$
Then for any $t\le T$,
\begin{equation}\lb{valjeu}
\begin{aligned}
\Phi(Y)_t+\int_0^t f(s, Y_s, \E[Y_s])ds&=\text{ess}\sup_{\tau\geq t}\text{ess}\inf_{\sigma\geq t}\E[\cL(\t,\s)|\cF_t]=\text{ess}\inf_{\sigma\geq t}\text{ess}\sup_{\tau\geq t}\E[\cL(\t,\s)|\cF_t].
\end{aligned}
\end{equation}As pointed out previously when $Y$ belongs to $\sp_c$ with $p>1$, then it belongs to $\cD$. Therefore, under assumptions (A1), the process $\Phi(Y)$ is continuous.
On the other hand, the second equality in \eqref{valjeu} is valid since by (A1)-(ii), (a)-(c), $h<g$ and the processes $(h(s,Y_s, \E[Y_s]))_{s\le T}$ and $(g(s,Y_s, \E[Y_s]))_{s\le T}$ belongs to $\sp_c$ since $Y$ belongs to $\sp_c$ (see e.g.
\cite{ehw}
for more details).
\ms

\nd Next let us define the martingale $M:=(M_t)_{0\leq t\leq T}$ as follows:
\begin{equation}
\begin{aligned}
M_t: &=\E\left\{\int_0^T\left[ |f(s,0,0)| +C_f(|Y_s|+\E|Y_s|)\right]ds+\sup_{s\le T}|g(s,0,0)|+\bt_1 \sup_{s\leq T}|Y_s|+\bt_2\sup_{s\leq T}\E|Y_s|\right.\\
&\left.\qquad\qquad\qquad\qquad+\sup_{s\le T}|h(s,0,0)|+\gm_1\sup_{s\leq T}|Y_s|+\gm_2\sup_{s\leq T}\E|Y_s|+|\xi|\ \ \big|\cF_t\right\}.
\end{aligned}
\end{equation} As $Y$ belongs to $\sp$ and by Assumptions (A1), the term inside the conditional expectation belongs to $L^p(d\p)$. As the filtration $\bF$ is Brownian then $M$ is continuous and by Doob's inequality with $p>1$ (\cite{revuzyor}, pp.54) one deduces that $M$ belongs also to $\sp$.
Next as $f$, $g$ and $h$ are Lipschitz, then by a linearization procedure of those functions one deduces that: $$
\begin{aligned}
|\E[\cL(\t,\s)|\cF_t]|\leq M_t
\end{aligned}
$$
for any $t\leq T$ and any stopping times $\sigma, \tau \in \mathcal{T}_t$. Then we obtain
$$
\begin{aligned}\forall t\le T,\,\,
|\Phi(Y)_t+\int_0^t f(s, Y_s, \E[Y_s])ds|\leq M_t.
\end{aligned}
$$
Therefore,
$$
\begin{aligned}
\E\{\sup_{t\leq T}|\Phi(Y)_t|^p \}\leq C_p \left\{ \E \left( \int_0^T| f(s, Y_s, \E[Y_s])|ds \right)^p +\E[\sup_{t\leq T}|M_t|^p] \right\},
\end{aligned}
$$
where $C_p$ is a positive constant that only depends on $p$ and $T$. It holds that $\Phi(Y)\in \cS_c^p$ since $Y\in \cS_c^p$ and $f$ is Lipschitz.
\end{proof}
Next we have the following result.
\begin{proposition}\label{fix_estimate_p>1}
Let Assumption (A1) holds for some $p>1$. If $\gm_1,\gm_2,\bt_1$ and $\bt_2$ satify
\begin{equation}\lb{cd1}
(\gm_1+\gm_2+\bt_1+\bt_2)^{\frac{p-1}{p}} \left[\left(\frac{p}{p-1}\right)^p(\gm_1+\bt_1)+(\gm_2+\bt_2)\right]^{\frac{1}{p}}<1,
\end{equation}
then there exists $\delta >0$ depending only on $p, C_f, \gm_1, \gm_2, \bt_1$ and $\bt_2$ such that $\Phi$ is a contraction on the time interval $[T-\delta, T]$.
\end{proposition}
\begin{proof} Let $Y, Y' \in \cS_c^p$. Then, for any $t \leq T$, we have,
\begin{equation}\label{phi1}
\begin{aligned}
&|\Phi(Y)_t-\Phi(Y')_t|\\&=| \esssup_{\tau\geq t}\essinf_{\sigma\geq t} \{\E\left[\int_t^{\sigma\wedge\tau}f(s,Y_s,\E[Y_s])ds +g( \s,Y_\sigma, \E[Y_t]_{t=\sigma})\mathbb{1}_{\{\sigma<\tau\}}\right.\\
&\left.\qquad +h( \t,Y_\tau, \E[Y_t]_{t=\tau})\mathbb{1}_{\{\tau\leq\sigma,\t<T\}}+\xi\mathbb{1}_{\{\tau=\sigma=T\}}|\cF_t\right]\}- \esssup_{\tau\leq t}\essinf_{\sigma\leq t}\{\E\left[\int_t^{\sigma\wedge\tau}f(s,Y'_s,\E[Y'_s])ds \right.\\&\qquad +g(\s, Y'_\sigma, \E[Y'_t]_{t=\sigma})\mathbb{1}_{\{\sigma<\tau\}}
\left.+h(\t, Y'_\tau, \E[Y'_t]_{t=\tau})\mathbb{1}_{\{\tau\leq\sigma,\t<T\}}+\xi\mathbb{1}_{\{\tau=\sigma=T\}}|\cF_t\right]\}|\\
&\leq \esssup_{\tau\geq t,\,\sigma\geq t}\E\left\{\int_t^{\sigma\wedge\tau}|f(s,Y_s,\E[Y_s])-f(s,Y'_s,\E[Y'_s])|ds\right.+|g( \s,Y_\sigma, \E[Y_t]_{t=\sigma})\\
&\left.\qquad \qquad -g(\s, Y'_\sigma, \E[Y'_t]_{t=\sigma})|\mathbb{1}_{\{\sigma<\tau\}}+|h(\t, Y_\tau, \E[Y_t]_{t=\tau})-h(\t,Y'_\tau, \E[Y'_t]_{t=\tau})|\mathbb{1}_{\{\tau\leq\sigma,\t<T\}} |\cF_t\right\}\\
&\leq \E\left\{\int_t^T|f(s,Y_s,\E[Y_s])-f(s,Y'_s,\E[Y'_s])|ds+(\bt_1+\gm_1)
\sup_{t\leq s\leq T}|Y_s-Y'_s||\cF_t\right\} \\&\qquad \qquad+(\bt_2+\gm_2)\sup_{t\leq s\leq T}\E[|Y_s-Y'_s|].
\end{aligned}
\end{equation}
Fix now $\delta>0$ and let $t\in [T-\delta,T]$. By the Lipschitz condition of $f$, \eqref{phi1} implies that
\begin{equation}\label{phi2}
\begin{aligned}
&|\Phi(Y)_t-\Phi(Y')_t|\\
&\leq\E\left[\delta C_f\{\sup_{T-\delta\leq s\leq T}|Y_s-Y'_s|+\sup_{T-\delta\leq s\leq T}\E[|Y_s-Y'_s|]\}\right.+(\bt_1+\gm_1)\sup_{T-\delta\leq s\leq T}|Y_s-Y'_s|\\
&\left. +(\bt_2+\gm_2)\sup_{T-\delta\leq s\leq T}\E[|Y_s-Y'_s|] \big|\cF_t\right]\\
&=(\delta C_f+\gm_1+\bt_1)\E\left[\sup_{T-\delta\leq s\leq T}|Y_s-Y'_s| \big| \cF_t\right]+(\delta C_f+\gm_2+\bt_2)\sup_{T-\delta\leq s\leq T}\E\{|Y_s-Y'_s|\}.
\end{aligned}
\end{equation}
As $p>1$, thanks to the convexity inequality
 $(ax_1+bx_2)^p \leq (a+b)^{p-1}(ax_1^p+bx_2^p)$ holding for any non-negative real constants $a, b, x_1$ and $x_2$, \eqref{phi2} yields
\begin{equation}\label{phi3}
\begin{aligned}
|\Phi(Y)_t-\Phi(Y')_t|^p&\leq (2\delta C_f+\gm_1+\gm_2+\bt_1+\bt_2)^{p-1} \left\{ (\delta C_f+\gm_1+\bt_1)\right. \\
&\left.\left(\E[\sup_{T-\delta\leq s\leq T}|Y_s-Y'_s| \big| \cF_t]\right)^p+(\delta C_f+\gm_2+\bt_2)\left(\E[\sup_{T-\delta\leq s\leq T}|Y_s-Y'_s|]\right)^p\right\}.
\end{aligned}
\end{equation}
Next, by taking expectation of the supremum over $t \in [T-\delta, T]$  on the both hand-sides of \eqref{phi3}, we have
 \begin{equation}\label{phi4}
\begin{aligned}
&\E\left[\sup_{T-\delta\leq s\leq T}|\Phi(Y)_s-\Phi(Y')_s|^p\right]\\
&\leq (2\delta C_f+\gm_1+\gm_2+\bt_1+\bt_2)^{p-1} \left\{ (\delta C_f+\gm_1+\bt_1)\right. \E[\left(\sup_{T-\delta\leq t\leq T}\left\{\E[\sup_{T-\delta\leq s\leq T}|Y_s-Y'_s||\cF_t]\right\}\right)^p]\\
&\left.\qquad\qquad +(\delta C_f+\gm_2+\bt_2)\left\{\E[\sup_{T-\delta\leq s\leq T}|Y_s-Y'_s|]\right\}^p\right\}.
\end{aligned}
\end{equation}
By applying Doob's inequality we have:
\begin{equation}\label{phi41}
\begin{aligned}\E[\left(\sup_{T-\delta\leq t\leq T}\left\{\E[\sup_{T-\delta\leq s\leq T}|Y_s-Y'_s||\cF_t]\right\}\right)^p]\le (\frac{p}{p-1})^p  \E\left[\sup_{T-\delta\leq s\leq T}|Y_s-Y'_s|^p\right]
\end{aligned}
\end{equation}
and by Jensen's one we have also
 \begin{equation}\label{phi51}
\left\{\E[\sup_{T-\delta\leq s\leq T}|Y_s-Y'_s|]\right\}^p\le
\E[\sup_{T-\delta\leq s\leq T}|Y_s-Y'_s|^p].
\end{equation}Plug now \eqref{phi41} and \eqref{phi51} in \eqref{phi4} to obtain:
 $$
 \|\Phi(Y)-\Phi(Y')\|_{\cS^p_c {([T-\delta, T])}} \leq \Lambda(C_f, p, \gm_1, \gm_2, \bt_1, \bt_2)(\d)\|Y-Y'\|_{\cS_c^p([T-\delta, T])}
$$
 where
\begin{align*}
\Lambda(C_f, p, \gm_1, \gm_2, \bt_1, \bt_2)(\d)&=(2 \delta C_f+\gm_1+\gm_2+\bt_1+\bt_2)^{\frac{p-1}{p}} \left[\left(\frac{p}{p-1}\right)^p(\delta C_f+\gm_1+\bt_1)\right.\\
&\left.+(\delta C_f+\gm_2+\bt_2)\right]^\frac{1}{p}.
\end{align*}
 Note that \eqref{cd1} is just $\Lambda(C_f, p, \gm_1, \gm_2, \bt_1, \bt_2) (0)<1$. Now as \\$\lim_{\d\rw 0} \Lambda(C_f, p, \gm_1, \gm_2, \bt_1, \bt_2) (\d)=\Lambda(C_f, p, \gm_1, \gm_2, \bt_1, \bt_2) (0)<1$, then there exists $\delta$ small enough which depends only on $C_f, p, \gm_1, \gm_2, \bt_1, \bt_2$ and not on $\xi$ nor $T$ such that $ \Lambda(C_f, p, \gm_1, \gm_2, \bt_1, \bt_2) (\d)<1$. It implies that $\Phi$ is a contraction on ${\cS_c^p([T-\delta, T])}$. Then there exists a process which belongs to ${\cS_c^p([T-\delta, T])}$ such that
 $$Y_t=\Phi(Y)_t, \,\,\forall t\in [T-\delta, T].$$
 \end{proof}
\nd We now show that the mean-field reflected BSDE \eqref{MFDRBSDE P>1} has a unique solution.
\begin{theorem}\label{fix_uniq}
Assume that Assumption (A1) holds for some $p>1$. If $\gm_1$ and $\gm_2$ satisfy
 \begin{equation}\label{fix_uniq1}
(\gm_1+\gm_2+\bt_1+\bt_2)^{\frac{p-1}{p}} \left[\left(\frac{p}{p-1}\right)^p(\gm_1+\bt_1)+(\gm_2+\bt_2)\right]^{\frac{1}{p}}<1,
\end{equation}
 then the mean-field doubly reflected BSDE \eqref{MFDRBSDE P>1} has a unique solution $(Y, Z, K^+,K^-)$.
\end{theorem}
\begin{proof}
Let $\d$ be as in Proposition \ref{fix_estimate_p>1}. Then there exists a process $Y\in {\cS_c^p([T-\delta, T])}$, which is the fixed point of $\Phi$ in this latter space and verifies:
For any $t\in [T-\delta,T]$,
$$
\begin{aligned}
Y_t &=\esssup_{\tau\geq t}\essinf_{\sigma\geq t} \E\left\{\int_t^{\sigma\wedge\tau}f(s,Y_s,\E[Y_s])ds +g( \s,Y_\sigma, \E[Y_t]_{t=\sigma})\mathbb{1}_{\{\sigma<\tau\}}\right.\\
&\left.\qquad\qquad\qquad\qquad\qquad+h(\t, Y_\tau, \E[Y_t]_{t=\tau})\mathbb{1}_{\{\tau\leq\sigma,\t<T\}}+\xi\mathbb{1}_{\{\tau=\sigma=T\}}|\cF_t\right\}.
\end{aligned}
$$
Next since $\xi \in L^p(d\p)$, $\E[(\int_0^T|f(s,\omega, Y_s,\E[Y_s])|ds)^p]<\infty$, the processes $(h(t,Y_t,\E[Y_t]))_{T-\d\le t\le T}$ and \\
$(g(t,Y_t,\E[Y_t]))_{T-\d\le t\le T}$ belong to ${\cS_c^p([T-\delta, T])}$
since $Y$ is so, and finally since $h<g$, then there exist processes $\bar Y\in {\cS_c^p([T-\delta, T])}$, $\bar Z\in \hlcd([T-\d,T])$ and $\bar K^{\pm}\in \ca ([T-\d,T])$ (see e.g. \cite{ehw} for more details) such that for any $t\in  [T-\d,T]$, it holds:
\begin{equation}\label{MFloc}
\left\{
\begin{aligned}
&\bar Y_t = \xi + \int_t^T f(s, Y_s, \E[Y_s])ds +\bar  K^+_T-\bar K^+_t - \bar K^-_T +\bar K^-_t -\int_t^T \bar Z_s dB_s;\\
&h(t,Y_t, \E[Y_t]) \leq \bar Y_t \leq g(t,Y_t, \E[Y_t]); \\
& \int_{T-\d}^T(\bar Y_s-h(s,Y_s, \E[Y_s]))d\bar K^+_s=0, \int_{T-\d}^T(\bar Y_s-g(s, Y_s, \E[Y_s]))d\bar K^-_s=0.
\end{aligned}\right.
\end{equation}
Therefore the process $\bar Y$ has the following representation: $\forall t\in [T-\d,T]$,
\begin{equation}\label{fix_con_value2}
\begin{aligned}
\bar Y_t &=\esssup_{\tau\geq t}\essinf_{\sigma\geq t} \E\left\{\int_t^{\sigma\wedge\tau}f(s,Y_s,\E[Y_s])ds +g(\s, Y_\sigma, \E[Y_t]_{t=\sigma})\mathbb{1}_{\{\sigma<\tau\}}\right.\\
&\left.\qquad\qquad\qquad\qquad\qquad+h( \t,Y_\tau, \E[Y_t]_{t=\tau})\mathbb{1}_{\{\tau\leq\sigma,\t<T\}}+\xi\mathbb{1}_{\{\tau=\sigma=T\}}|\cF_t\right\}.
\end{aligned}
\end{equation}
It follows that for any $t\in [T-\d,T]$, $Y_t=\bar Y_t$. Thus $(Y,\bar Z, \bar K^{\pm})$ verifies \eqref{MFDRBSDE P>1} and \eqref{MFloc} on $[T-\d,T]$, i.e., for $t\in [T-\d,T]$
\begin{equation}\label{MFloc2}
\left\{
\begin{aligned}
& Y_t = \xi + \int_t^T f(s, Y_s, \E[Y_s])ds +\bar  K^+_T-\bar K^+_t - \bar K^-_T +\bar K^-_t -\int_t^T \bar Z_s dB_s;\\
&h( t,Y_t, \E[Y_t]) \leq Y_t \leq g(t,Y_t, \E[Y_t]); \\
& \int_{T-\d}^T( Y_s-h(s,Y_s, \E[Y_s]))d\bar K^+_s=0, \int_{T-\d}^T( Y_s-g(s, Y_s, \E[Y_s]))d\bar K^-_s=0.
\end{aligned}\right.
\end{equation}
But $\d$ of Proposition \eqref{fix_estimate_p>1} does not depend on the terminal condition $\xi$ nor on $T$, therefore there exists another process $Y^1$ which is a fixed point of $\Phi$ in ${\cS_c^p([T-2\delta, T-\delta])}$ with terminal condition $Y_{T-\d}$, i.e., for any $t\in [T-2\delta, T-\delta]$,
\begin{equation}\label{fix_con_value2x}
\begin{aligned}
Y^1_t &=\esssup_{\tau\in [t,T-\d]}\essinf_{\sigma\in [t,T-\d]} \E\left\{\int_t^{\sigma\wedge\tau}f(s,Y^1_s,\E[Y^1_s])ds +g(\s, Y^1_\sigma, \E[Y^1_t]_{t=\sigma})\mathbb{1}_{\{\sigma<\tau\}}\right.\\
&\left.\qquad\qquad\qquad+h(\t,Y^1_\tau, \E[Y^1_t]_{t=\tau})\mathbb{1}_{\{\tau\leq\sigma,\t<T-\d\}}+Y_{T-\d}\mathbb{1}_{\{\tau=\sigma=T-\d\}}|\cF_t\right\}.
\end{aligned}
\end{equation}
Then as previously, there exist processes $(\bar Z^1, \bar K^{1,\pm})$ ($\bar K^{1,\pm}\in \ca([T-2\d,T-\d]))$ such that $(Y^1,\bar Z^1, \bar K^{1,\pm})$ verify: For any $t\in [T-2\d,T-\d]$,
\begin{equation}\label{MFloc3}
\left\{
\begin{aligned}
& Y^1_t = Y_{T-\d} + \int_t^{T-\d}  f(s, Y^1_s, \E[Y^1_s])ds +\bar  K^{1,+}_{T-\d} -\bar K^{1,+}_t - \bar K^{1,-}_{T-\d}  +\bar K^{1,-}_t -\int_t^{T-\d}  \bar Z^1_s dB_s;\\
&h(t, Y^1_t, \E[Y^1_t]) \leq Y^1_t \leq g(t,Y^1_t, \E[Y^1_t]); \\
& \int_{T-2\d}^{T-\d} ( Y^1_s-h(Y^1_s, \E[Y^1_s]))d\bar K^{1,+}_s=0, \int_{T-2\d}^{T-\d} ( Y^1_s-g( Y^1_s, \E[Y^1_s]))d\bar K^{1,-}_s=0.
\end{aligned}\right.
\end{equation}
Concatenating now the solutions $(Y,\bar Z, \bar K^{\pm})$ and $(Y^1,\bar Z^1, \bar K^{1,\pm})$ we obtain a solution of \eqref{MFDRBSDE P>1} on
$[T-2\d,T]$. Actually for $t\in[T-2\delta,T]$, let us set:
\begin{align*}
&\tl Y_t=Y_t1_{[T-\delta,T]}(t)+{Y}^1_t1_{[T-2\delta,T-\delta)}(t),\\
&\tl Z_t=\bar Z_t1_{[T-\delta,T]}(t)+\bar{Z}^1_t1_{[T-2\delta,T-\delta)}(t),\\
&\dis\int_{T-2\delta}^td\tl K_t^{1,\pm}=\dis\int_{T-2\delta}^t\lbrace 1_{[T-\delta,T]}(s)d\bar K_s^{0,\pm}+1_{[T-2\delta,T-\delta]}(s)d\bar{K}_s^{1,\pm} \rbrace.
\end{align*}Then $\tl Y\in {\cS_c^p([T-2\delta, T]}$, $\tl Z\in \hlcd([T-2\d,T])$ and $\tl K^{\pm}\in \ca ([T-2\d,T])$ and they verify: For any $t\in [T-2\delta, T]$,
\begin{equation}\label{MFDRBSDE P>1x}
\left\{
\begin{aligned}&\tl Y_t = \xi + \int_t^T f(s, \tl Y_s, \E[\tl Y_s])ds +\tl K^+_T-\tl K^+_t - \tl K^-_T +\tl K^-_t -\int_t^T \tl Z_s dB_s;\\
&h(t, \tl Y_t, \E[\tl Y_t]) \leq \tl Y_t \leq g(t,\tl Y_t, \E[\tl Y_t]); \\
&\int_{T-2\d}^T(\tl Y_s-h(s,\tl Y_s, \E[\tl Y_s]))d\tl K^+_s=0 \mx{ and } \int_{T-2\d}^T(\tl Y_s-g(s, \tl Y_s, \E[\tl Y_s]))d\tl K^-_s=0.
\end{aligned}\right.
\end{equation}
But we can do the same on $[T-3\d,T-2\d]$, $[T-4\d,T-3\d]$, etc. and at the end, by concatenation of those solutions, we obtain a solution $(Y,Z,K^{\pm})$ which satisfies \eqref{MFDRBSDE P>1}.

Let us now focus on uniqueness. Assume there is another solution $(\underbar Y,\underbar Z,\underbar K^{\pm})$ of \eqref{MFDRBSDE P>1}. It means that $\underbar Y$ is a fixed point of $\Phi$ on
${\cS_c^p([T-\delta, T])}$, therefore for any $t\in [T-\d,T]$, $Y_t=\underbar Y_t$.  Next writing equation \eqref{MFDRBSDE P>1} for $Y$ and $\underbar Y$ on $[T-2\d,T-\d]$, using the link with zeros-sum Dynkin games (see Lemma \ref{lemma31}) and finally the uniqueness of the fixed point of $\Phi$ on ${\cS_c^p([T-2\delta, T-\d])}$ to obtain that for any $t\in [T-2\d,T-\d]$, $Y_t=\underbar Y_t$.  By continuing this procedure on $[T-3\d,T-2\d]$, $[T-4\d,T-3\d]$, etc. we obtain that $Y=\underline Y$. The equality between the stochastic integrals implies that $Z=\underline Z$. Finally as $h<g$ and since $Y=\underbar Y$, then $K^+=\underbar K^+$ and $K^-=\underbar K^-$ (see e.g. \cite{ehw}) for more details. Thus the solution is unique.
The proof is complete.
\end{proof}
\begin{remark}\label{gene} There is no specific difficulty to consider the following more general framework of equations \eqref{MFDRBSDE P>1} and \eqref{MFDRBSDE P=1}.
$$
\left\{
\begin{aligned}
&Y_t = \xi + \int_t^T f(s, Y_s, \p_{Y_s})ds +K^+_T-K^+_t - K^-_T +K^-_t -\int_t^T Z_s dB_s,\quad 0\leq t\leq T;\\
&h( t,Y_t, \p_{Y_s}) \leq Y_t \leq g( t,Y_t, \p_{Y_s}), \quad \forall t\in [0,T]; \\
&\mbox{ and }\int_0^T(Y_s-h(s,Y_s, \p_{Y_s}))dK^+_s=0, \int_0^T(Y_s-g(s, Y_s, \p_{Y_s}))dK^-_s=0
\end{aligned}\right.
$$
where the Lipschitz property of $f$, $h$ and $g$ w.r.t. $\p_{Y_t}$ should be read as: for $\Psi \in \{f,g,h\}$ for any $\nu,\nu'$ probabilities
$$
|\Psi(\nu)-\Psi(\nu')|\le Cd_p(\nu,\nu')
$$
where $d_p(.,.)$ is the $p$-Wasserstein distance on the subset $\mathcal{P}_p(\R)$ of probability measures with finite $p$-th moment, formulated in terms of a coupling between two random variables $X$ and $Y$ defined on the same probability space:
\begin{equation*}\label{d-coupling}
d_p(\mu,\nu):=\inf\left\{\left(\E\left[|X-Y|^p\right]\right)^{1/p},\,\,\text{law}(X)=\mu,\,\text{law}(Y)=\nu \right\}.\qed
\end{equation*}
\end{remark}
\subsection{The case p=1}${}$\\
We proceed as we did in the case when $p>1$.
We have the following result.
\begin{proposition}\label{contaction p=1}
Let Assumptions (A1) hold for some $p=1$. If $\gm_1,\gm_2,\bt_1$ and $\bt_2$ satisfy
\begin{equation}\label{cdtionsurd}
\gm_1+\gm_2+\bt_1+\bt_2<1,
\end{equation}
then there exists $\delta >0$ depending only on $C_f, \gm_1, \gm_2, \bt_1, \bt_2$ such that $\Phi$ is a contraction on the space $\cD([T-\delta, T])$.
\end{proposition}
\begin{proof} Let $\d$ be a positive constant and $\th$ a stopping time which belongs to $[T-\d,T]$. Therefore
$$
\begin{aligned}
&|\Phi(Y)_\th-\Phi(Y')_\th|=
| \esssup_{\tau\geq \th}\essinf_{\sigma\geq \th} \{\E\left[\int_\th^{\sigma\wedge\tau}f(s,Y_s,\E[Y_s])ds +g( \s,Y_\sigma, \E[Y_t]_{t=\sigma})\mathbb{1}_{\{\sigma<\tau\}}\right.\\
&\left.\qquad +h(\t, Y_\tau, \E[Y_t]_{t=\tau})\mathbb{1}_{\{\tau\leq\sigma,\t<T\}}+\xi\mathbb{1}_{\{\tau=\sigma=T\}}|\cF_t\right]\}- \esssup_{\tau\geq \th}\essinf_{\sigma\geq \th}\{\E\left[\int_t^{\sigma\wedge\tau}f(s,Y'_s,\E[Y'_s])ds \right.\\&\qquad +g(\s, Y'_\sigma, \E[Y'_t]_{t=\sigma})\mathbb{1}_{\{\sigma<\tau\}}
\left.+h(\t, Y'_\tau, \E[Y'_t]_{t=\tau})\mathbb{1}_{\{\tau\leq\sigma,\t<T\}}+\xi\mathbb{1}_{\{\tau=\sigma=T\}}|\cF_\th\right]\}|
\\&\leq \esssup_{\tau\geq \th}\esssup_{\sigma\geq \th}\E\left\{\int_\th^{\sigma\wedge\tau}|f(s,Y_s,\E[Y_s])-f(s,Y'_s,\E[Y'_s])|ds\right.+|g(\s, Y_\sigma, \E[Y_t]_{t=\sigma})\\
&\left.\qquad-g(\s, Y'_\sigma, \E[Y'_t]_{t=\sigma})|\mathbb{1}_{\{\sigma<\tau\}}+|h( \t,Y_\tau, \E[Y_t]_{t=\tau})-h(\t, Y'_\tau, \E[Y'_t]_{t=\tau})|\mathbb{1}_{\{\tau\leq\sigma,\t<T\}} |\cF_\th\right\}\\&\leq \E\left\{\int_{T-\d}^{T}|f(s,Y_s,\E[Y_s])-f(s,Y'_s,\E[Y'_s])|ds|\cF_\th\right\}
+\esssup_{\sigma\geq \th}\E\{|g(\s, Y_\sigma, \E[Y_t]_{t=\sigma})-g(\s, Y'_\sigma, \E[Y'_t]_{t=\sigma})||\cF_\th\}\\&\qquad +\esssup_{\t\geq \th}\E\{|h(\t,Y_\tau, \E[Y_t]_{t=\tau})-h(\t, Y'_\tau, \E[Y'_t]_{t=\tau})| |\cF_\th\}.
\end{aligned}
$$Taking now expectation on both hand-sides to obtain:
$$
\begin{aligned}
\E[|\Phi(Y)_\th-\Phi(Y')_\th|]&\leq 2\d C_f\sup_{\t \in [T-\d,T]}\E[|Y_\t-Y'_\t]+\sup_{\sigma\geq \th}\E\{|g(\s, Y_\sigma, \E[Y_t]_{t=\sigma})-g(\s, Y'_\sigma, \E[Y'_t]_{t=\sigma})|\}\\&\qquad +\sup_{\t\geq \th}\E\{|h( \t,Y_\tau, \E[Y_t]_{t=\tau})-h(\t, Y'_\tau, \E[Y'_t]_{t=\tau})|\}\\
&\le 2\d C_f\sup_{\t \in [T-\d,T]}\E[|Y_\t-Y'_\t]+\sup_{\sigma\geq \th}\E\{|g(\s, Y_\sigma, \E[Y_t]_{t=\sigma})-g(\s,Y'_\sigma, \E[Y'_t]_{t=\sigma})|\}\\&\qquad+\sup_{\t\geq \th}\E\{|h(\t, Y_\tau, \E[Y_t]_{t=\tau})-h( \t,Y'_\tau, \E[Y'_t]_{t=\tau})|\}.
\end{aligned}
$$
Then for any stopping time $\th$ valued in $[T-\d,T]$, we have:
$$
\E[|\Phi(Y)_\th-\Phi(Y')_\th|]\leq \underbrace{(2\d C_f+\bt_1+\bt_2+\gm_1+\gm_2)}_{\Sigma(\d)}\sup_{\t \in [T-\d,T]}\E[|Y_\t-Y'_\t|].$$

Next since $\bt_1+\bt_2+\gm_1+\gm_2 <1 $, then for $\d$ small enough we have $\Sigma(\d)<1$ ($\d$ depends neither on $\xi$ nor on $T$) and $\Phi$ is a contraction on the space
$\cD([T-\d,T])$. Therefore it has a fixed point $Y$, which then verifies:
\begin{equation}\lb{repy}
\begin{aligned}
&Y\in \cD([T-\d,T]) \mbox{ and }\forall t\in [T-\d,T], \\&\qquad \qquad Y_t=\esssup_{\tau\geq t} \essinf_{\sigma\geq t}\{\E\{\int_t^{\sigma\wedge\tau}f(s,Y_s,\E[Y_s])ds +g(\s, Y_\sigma, \E[Y_t]_{t=\sigma})\mathbb{1}_{\{\sigma<\tau\}}\\
&\qquad \qquad\qquad \qquad +h(\t, Y_\tau, \E[Y_t]_{t=\tau})\mathbb{1}_{\{\tau\leq\sigma, \t<T\}}+\xi\mathbb{1}_{\{\tau=\sigma=T\}}|\cF_t\}\}.
\end{aligned}
\end{equation}\end{proof}
As a by-product we have the following result which stems from the link between the value of a zero-sum Dynkin game and doubly reflected BSDE given in \eqref{eq:jeux}.
\begin{corollary}
Let Assumption (A1) hold for some $p=1$. If $\gm_1,\gm_2,\bt_1$ and $\bt_2$ satisfy \eqref{cdtionsurd} then there exists $\delta >0$, depending only on $C_f, \gm_1, \gm_2, \bt_1, \bt_2$, and $\cP$-measurable processes $Z^0$, $K^{0,\pm}$ such that:
\begin{equation}\label{eq:jeux2}\left\{\begin{array}{l}
\p-a.s.,\,\,\int_{T-\d}^T|Z^0_s|^2ds<\infty;\,\,K^{0,\pm} \in \ca \mbox{ and }K^{0,\pm}_{T-\d}=0;\\\\
Y_t=\xi+\int_t^Tf(s,Y_s,\E[Y_s])ds+K^{0,+}_T-K^{0,+}_t-K^{0,-}_T+K^{0,-}_t-\int_t^T Z^0_sdB_s,\,\,T-\d\le t\le T; \\\\
h(t,Y_t,\E[Y_t])\le Y_t\le g(t,Y_t,\E[Y_t]),\,\,T-\d\le t\le T;\\\\
\int_{T-\d}^T(Y_t-h(t,Y_t,\E[Y_t]))d
K^{0,+}_t=\int_{T-\d}^T(Y_t-g(t,Y_t,\E[Y_t]))d
K^{0,-}_t=0.\end{array}\right.
\end{equation}
\end{corollary}
We now give the main result of this subsection.
\begin{theorem}
Let $f, h$, $g$ and $\xi$ satisfying Assumption (A1) for $p=1$. Suppose that
 \begin{equation}\lb{condthm38}
\gm_1+\gm_2+\bt_1+\bt_2<1.
\end{equation}
Then, there exist $\cP$-mesurable processes $(Y,Z,K^\pm)$ unique solution of the mean-field reflected BSDE \eqref{MFDRBSDE P=1}, i.e.,
\begin{equation}\label{MFDRBSDE P=1xx}
\left\{
\begin{aligned}
&Y \in \cD,\quad  Z \in\hlcd \quad and\quad K^+, K^- \in \ca;\\
&Y_t = \xi + \int_t^T f(s, Y_s, \E[Y_s])ds +K^+_T-K^+_t - K^-_T +K^-_t -\int_t^T Z_s dB_s,\quad 0\leq t\leq T;\\
&h( t,Y_t, \E[Y_t]) \leq Y_t \leq g(t, Y_t, \E[Y_t]), \quad \forall t\in [0,T] ;\\
&\int_0^T(Y_s-h( s,Y_s\mx{ \emph{and} } \E[Y_s]))dK^+_s=0, \int_0^T(Y_s-g(s, Y_s, \E[Y_s]))dK^-_s=0.\\
\end{aligned}\right.
\end{equation}
\end{theorem}
\begin{proof} Let $\d$ be as in Proposition \ref{contaction p=1} and $Y$ the fixed point of $\Phi$ on $\cD([T-\d,T])$ which exists since \eqref{cdtionsurd} is satisfied. Next let  $Y^1$ be the fixed point of $\Phi$ on ${\cD([T-2\delta, T-\delta])}$ with terminal condition $Y_{T-\d}$, i.e., for any $t\in [T-2\delta, T-\delta]$,
\begin{equation}\label{fix_con_value2xx}
\begin{aligned}
Y^1_t &=\esssup_{\tau\in [t,T-\d]}\essinf_{\sigma\in \in [t,T-\d]} \E\left\{\int_t^{\sigma\wedge\tau}f(s,Y^1_s,\E[Y^1_s])ds +g(\s, Y^1_\sigma, \E[Y^1_t]_{t=\sigma})\mathbb{1}_{\{\sigma<\tau\}}\right.\\
&\left.\qquad\qquad\qquad+h(\t, Y^1_\tau, \E[Y^1_t]_{t=\tau})\mathbb{1}_{\{\tau\leq\sigma,\t<T-\d\}}+Y_{T-\d}\mathbb{1}_{\{\tau=\sigma=T-\d\}}|\cF_t\right\}.
\end{aligned}
\end{equation}The process $Y^1$ exists since condition \eqref{cdtionsurd} is satisfied and $\d$ depends neither on $T$ nor on the terminal condition. Once more the link between reflected BSDEs and zero-sum Dynkin games (see Lemma \ref{lemma31}) implies the existence of $\cP$-measurable processes $Z^1$, $K^{1,\pm}$ such that:
\begin{equation}\label{eq:jeux3}\left\{\begin{array}{l}
\p-a.s.,\,\,\int_{T-2\d}^{T-\d}|Z^1_s|^2ds<\infty;\,\,K^{1,\pm} \in \ca \mbox{ and }K^{1,\pm}_{T-2\d}=0;\\\\
Y^1_t=Y_{T-\d}+\int_t^{T-\d}f(s,Y^1_s,\E[Y^1_s])ds+K^{1,+}_{T-\d}-K^{1,+}_t-K^{1,-}_{T-\d}+K^{1,-}_t-\int_t^{T-\d} Z^1_sdB_s,\,t\in [T-2\d,T-\d]; \\\\
h(t,Y^1_t,\E[Y^1_t])\le Y^1_t\le g(t,Y^1_t,\E[Y^1_t]),\,\,t\in [T-2\d,T-\d];\\\\
\int_{T-2\d}^{T-\d}(Y^1_t-h(t,Y^1_t,\E[Y^1_t]))d
K^{1,+}_t=\int_{T-2\d}^{T-\d}(Y^1_t-g(t,Y^1_t,\E[Y^1_t]))d
K^{1,-}_t=0.\end{array}\right.
\end{equation}
Concatenating now the solutions $(Y,Z^0,K^{0,\pm})$ of \eqref{eq:jeux2} and $(Y^1,Z^1,K^{1,\pm})$ we obtain a solution of \eqref{MFDRBSDE P=1} on
$[T-2\d,T]$. Actually for $t\in[T-2\delta,T]$, let us set:
\begin{align*}
&\tl Y_t=Y_t1_{[T-\delta,T]}(t)+{Y}^1_t1_{[T-2\delta,T-\delta)}(t),\\
&\tl Z_t=Z^0_t1_{[T-\delta,T]}(t)+{Z}^1_t1_{[T-2\delta,T-\delta)}(t),\\
&\dis\int_{T-2\delta}^td\tl K_t^{1,\pm}=\dis\int_{T-2\delta}^t\lbrace 1_{[T-\delta,T]}(s)d K_s^{0,\pm}+1_{[T-2\delta,T-\delta]}(s)d{K}_s^{1,\pm} \rbrace.
\end{align*}Then $\tl Y\in {\cD([T-2\delta, T]}$, $\tl Z\in \hlcd([T-2\d,T])$ and $\tl K^{\pm}\in \ca ([T-2\d,T])$ and they verify: For any $t\in [T-2\delta, T]$,
\begin{equation}\label{MFDRBSDE P>1xxx}
\left\{
\begin{aligned}&\tl Y_t = \xi + \int_t^T f(s, \tl Y_s, \E[\tl Y_s])ds +\tl K^+_T-\tl K^+_t - \tl K^-_T +\tl K^-_t -\int_t^T \tl Z_s dB_s;\\
&h( t,\tl Y_t, \E[\tl Y_t]) \leq \tl Y_t \leq g(t,\tl Y_t, \E[\tl Y_t]); \\
&\int_{T-2\d}^T(\tl Y_s-h(s,\tl Y_s, \E[\tl Y_s]))d\tl K^+_s=0 \mx{ \emph{and} } \int_{T-2\d}^T(\tl Y_s-g(s, \tl Y_s, \E[\tl Y_s]))d\tl K^-_s=0.
\end{aligned}\right.
\end{equation}
But we can do the same on $[T-3\d,T-2\d]$, $[T-4\d,T-3\d]$, etc. and at the end, by concatenation of those solutions, we obtain a solution $(Y,Z,K^{\pm})$ which satisfies \eqref{MFDRBSDE P>1}.

Let us now focus on uniqueness. Assume there is another solution $(\underbar Y,\underbar Z,\underbar K^{\pm})$ of \eqref{MFDRBSDE P>1}. It means that $\underbar Y$ is a fixed point of $\Phi$ on
${\cD([T-\delta, T])}$, therefore for any $t\in [T-\d,T]$, $Y_t=\underbar Y_t$.  Next writing equation \eqref{MFDRBSDE P>1} for $Y$ and $\underbar Y$ on $[T-2\d,T-\d]$, using the link with zeros-sum Dynkin games (Lemma \ref{lemma31}), the uniqueness of the fixed point of $\Phi$ on ${\cD([T-2\delta, T-\d])}$ implies that for any
$t\in [T-2\d,T-\d]$, $Y_t=\underbar Y_t$.  By continuing this procedure on $[T-3\d,T-2\d]$, $[T-4\d,T-3\d]$, etc. we obtain that $Y=\underline Y$. The equality between the stochastic integrals implies that $Z=\underline Z$. Finally as $h<g$ and since $Y=\underbar Y$, then $K^+=\underbar K^+$ and $K^-=\underbar K^-$ (see e.g. \cite{ehw} for more details). Thus the solution is unique.
The proof is complete.

Finally let us notice that the same Remark \ref{gene} is valid for this case $p=1$.\end{proof}
\section{\Large{Penalty method for the mean-field doubly reflected BSDE}}
\nd In this section, we study doubly reflected BSDE with mean-field term involved in the coefficients and obstacles by using the penalty method. Actually we consider the following mean-field reflected BSDE:
\begin{equation}\label{MF-DRBSDE}\left\{
\begin{aligned}
&Y\in \mathcal{S}^2_c, Z\in \mathcal{H}^{2,d}, K^\pm\in \mathcal{S}^2_{ci};\\
&Y_t= \xi +\int_t^T f(s,Y_s,\E[Y_s],Z_s)ds+ K_T^+-K_t^+-(K_T^- -K_t^-)-\int_t^T Z_s dB_s,\,\, t\leq T;\\
&h(t,Y_t,E[Y_t])\leq Y_t\leq g(t,Y_t,E[Y_t]),\ \  t\leq T;\\
&\int_0^T\left(Y_t-h(t,Y_t,E[Y_t])\right)dK_t^+=\int_0^T\left(g(t,Y_t,E[Y_t])-Y_t\right)dK_t^-=0.
\end{aligned}\right.
\end{equation}
\par
\nd For the ease of exposition, we assume that $p=2$ in \textbf{Assumption (A1)}. On the other hand, now $f$ may depend on $z$. Precisely the assumptions in this section are the following ones.\par
\ms

\noindent\textbf{\underline{Assumption (A2)}}\\ \\
\noindent (i) (a) $f(t,y,y',z)$ is Lipschitz with respect to $(y,y',z)$ uniformly in $(t,\omega)$, i.e. there exists a constant $C$ that $\p$-a.s. for all $t\in[0,T]$ and $y_1,y_1',y_2,$  $y_2'\in R$,
$$|f(t,y_1,y_1',z)-f(t,y_1,y_1',z)|\leq C(|y_1-y_2|+|y_1'-y_2'|+|z_1-z_2|), $$ and $(f(t,0,0,0))_{t\leq T}\in \mathcal{H}^{2,1}.$

(b) $y'\mapsto f(t,y,y',z)$ is non-decreasing for fixed $t,y,z$.

(c) $y\mapsto f(t,y,y',z)$ is non-decreasing for fixed $t,y,z$.
\par
\ms

\noindent (ii) (a) $\p-a.s.$ for any $t\le T$, $h(\omega,t,y,y')$ and $g(\omega,t,y,y')$ are non-decreasing w.r.t $y$ and $y'$;

(b) The processes $(h(\omega,t,0,0)_{t\le T}$ and
$(g(\omega,t,0,0)_{t\le T}$ belong to $\cS^2_c$;

(c) $\p$-a.s.  $|g(t,\omega,y,y')-g(t,\omega,0,0)|+|h(t,\omega,y,y')-h(t,\omega,0,0)|\le C(1+|y|+|y'|)$ for some constant $C$ and any $t,y, y'$;
\ms

(d) \underline{Adapted Mokobodski's condition}: There exists a process $(X_t)_{t\le T}$ satisfying
$$ \forall t\le T,\,\,X_t=X_0+\int_0^t J_s dB_s +V_t^+-V_t^-$$ with $J\in\mathcal{H}^{2,d}$ and  $V^+,V^-\in \mathcal{S}_{ci}^2$ such that $\p$-a.s., for any $t,y,y'$,
\begin{equation*}
\begin{aligned}
h(\omega,t,y,y')\leq X_t\leq g(\omega,t,y,y');
\end{aligned}
\end{equation*}
(e) $\p$-a.s., $h(t,\og,y,y')<g(t,\og,y,y'),$ for any $t, y, y'\in \R$.\par
\noindent(iii) $\xi$ is an $\cF_T$- measurable, $\R$-valued r.v., $\E[\xi^2]<\infty$ and satisfies $\p$-a.s., $h(T,\xi, \E[\xi])\leq \xi \leq g(T,\xi, \E[\xi])$.\par
\ms

We denote that, the constant $C$ changes line by line in this paper.\par
\begin{remark}\label{remark for assumption}

(a) Assumption (A2)-(i),(c) is not stringent since if $f$ does not satisfy this property one can modify the mean-field doubly reflected BSDE \eqref{MF-DRBSDE} by using an exponential transform in such a way to fall in the case where $f$ satisfies
(A2)-(i),(c) (one can see e.g. \cite{mfrbsde2019}, Corollary 4.1 for more details).
\ms

\nd (b) In this penalization method, contrarily to the previous one, we cannot replace in $f(.)$ the expectation $\E[Y_s]$ with the law $\p_{Y_s}$, because of lack of comparison in this latter case. For the same reason $f(.)$ cannot depend on $\E[Z_s]$ or more generally on $\p_{Z_s}$, the law of $Z_s$.
\ms

\noindent (c) On the adapted Mokobodski's condition: We use Assumption (A2)-(ii),(d) to let
$h(t,Y_t,\E[Y_t])\leq Y_t\leq g(t,Y_t,\E[Y_t])$, $t\leq T$, possible. When we do not have Assumption (A2)-(ii),(d), a counter-example is considered below:\\
\ms
\nd For simplicity, we set $g(t,y,y')=\bar{C}$ where $\bar{C}$ is a large enough constant in such a way that \eqref{MF-DRBSDE} becomes a RBSDE, i.e.
\begin{equation}\label{example RBSDE}\left\{
\begin{aligned}
&Y_t= \xi +\int_t^T f(s,Y_s,\E[Y_s],Z_s)ds+ K_T-K_t-\int_t^T Z_s dB_s,\,t\leq T;\\
& Y_t\geq h(t,Y_t,\E[Y_t]),\,t\leq T;\,\,
\int_0^T\left(Y_t-h(t,Y_t,\E[Y_t])\right)dK_t=0.
\end{aligned}\right.
\end{equation}
Then, let $Y^K$ satisfy
\begin{equation}\label{YK for general example}
\begin{aligned}
Y^K_t= \xi +\int_t^T f(s,Y^K_s,\E[Y^K_s],Z^K_s)ds+ K_T-K_t-\int_t^T Z^K_s dB_s, \,\,t\le T,
\end{aligned}
\end{equation}
for a given increasing process K. If there exists a solution $Y$ of \eqref{example RBSDE}, then there exists an increasing process $K$ such that $Y_t\geq h(t,Y_t,\E[Y_t])$. Now if we do not have Assumption (A2)-(ii),(d), then, for some $h$ (see a explicit example below) and a fixed $t$, we obtain that
$$ Y^K_t-h(t,Y^K_t,\E[Y^K_t])<0,\ \ \ \text{for any increasing processes} \ \ K,$$
which means \eqref{example RBSDE} has no solution.\par More precisely, when we take $h(t,y,y')=y+y'+1$, $\xi=-T$ and $f=1$, for $T=1$, it follows that
\begin{equation}\label{example: contradiction on T=1 YK}
\begin{aligned}
&Y^K_t=-t+K_1-K_t-\int_t^1 Z^K_s dB_s,\ \ \text{for an increasing process } K. \\
\end{aligned}
\end{equation}
  Then for $t\in[0,1)$, we have
\begin{equation}
\begin{aligned}
\E[Y^K_t]=-t +\E[K_1-K_t]\geq -t> -1,
\end{aligned}
\end{equation}
which does not satisfy the inequality $\E[Y^K_t]\leq -1 $.
Therefore, corresponding Mean-field RBSDE, i.e.,
\begin{equation}\label{example: contradiction on T=1 Y}
\left\{\begin{aligned}
&Y_t=-t+K_1-K_t-\int_t^1 Z_s dB_s,\\
& Y_t\geq Y_t+\E[Y_t]+1, \ \ \text{and}\ \ \int_0^1\left(\E[Y_t]+1\right)dK_t=0,
\end{aligned}\right.
\end{equation}
has no solution, because for $t\in [0,1)$, $Y_t$ does not satisfy the obstacle constraint for any increasing process $K$.
\par
However, if we have \emph{Assumption (A2)-}\emph{(ii),(d)}, i.e., for any $t\in[0,T)$,
$$\sup_{y,y'} h(t,y,y')\leq X_t,$$
then, for any $h$ and  $t\in[0,T)$, there exists an increasing $K$ such that $Y^K_t$ satisfying
$$Y^K_t\geq X_t \geq \sup_{y,y'} h(t,y,y') \geq h(t,Y^K_t,\E[Y^K_t]).$$
Thus $Y$ satisfying $Y_t\geq h(t,Y_t,\E[Y_t])$, $t\le T$, is possible in \eqref{example RBSDE} for any $h$ satisfying \emph{Assumption (A2)-}\emph{(ii),(d)} in this case.\qed
\end{remark}

Now we are going to show that \eqref{MF-DRBSDE} has a solution by penalization. First, we define the process $Y^{0,0}$ by the solution of a standard Mean-field BSDE, i.e.
\begin{equation}\label{MF-BSDE for 00}
\begin{aligned}
Y_t^{0,0}= \xi +\int_t^T f(s,Y_s^{0,0},\E[Y_s^{0,0}],Z_s^{0,0})ds-\int_t^T Z_s^{0,0} dB_s,\,t\le T,
\end{aligned}
\end{equation}
whose existence and uniqueness is obtained thanks to \cite{li2018general}*{Theorem 3.1}. By assuming that
\begin{equation}\label{notation of n=-1 m=-1}
Y^{n,-1}=Y^{n,0}, \ \ Y^{-1,m}=Y^{0,m},\ \ \text{and}  \ \ Y^{-1,-1}=Y^{0,0}, \ \ \text{for any}\ \ n\geq 0, m\geq 0,
\end{equation}  \\
we present a series of mean-field BSDEs, for $n,m \geq 0$, as follows:
\begin{equation}\label{MF-DRBSDE for nm}\left\{
\begin{aligned}
&Y^{n,m}\in\cS^2_c,\ \ Z^{n,m}\in \cH^{2,d};\\
&Y_t^{n,m}= \xi +\int_t^T f(s,Y_s^{n,m},\E[Y_s^{n-1,m-1}],Z_s^{n,m})ds+ K_T^{n,m,+}-K_t^{n,m,+}-(K_T^{n,m,-}-K_t^{n,m,-})-\int_t^T Z_s^{n,m} dB_s,t\le T,\\
\end{aligned}\right.
\end{equation}
where  for any $t\le T$,
\begin{equation}\label{expknm}
\begin{aligned}
&K_t^{n,m,+}:=m\int_0^t \left(Y_s^{n,m}-h(s,Y_s^{n-1,m-1},\E[Y_s^{n-1,m-1}])\right)^- ds \mx{ and }\\
&K_t^{n,m,-}:=n\int_0^t \left(Y_s^{n,m}-g(s,Y_s^{n-1,m-1},\E[Y_s^{n-1,m-1}])\right)^+ ds.
\end{aligned}
\end{equation}
By comparison theorem of mean-field BSDE (see \cite{li2018general}*{Theorem 2.2}), we have the following inequality for any $n\geq 0, m\geq 0$, $t\in\time$,
\begin{equation}\label{inequality of Ynm Yn+1m and Ynm+1}
\begin{aligned}
Y_t^{n+1,m}\leq Y_t^{n,m}\leq Y_t^{n,m+1}.
\end{aligned}
\end{equation}
It is obtained by induction on $k$ such that $n+m\leq k$. For $k=0$ the property holds true since by comparison we have: For any $t\le T$,
\begin{equation}\label{ineq1}
\begin{aligned}
Y_t^{1,0}\leq Y_t^{0,0}\leq Y_t^{0,1}.
\end{aligned}
\end{equation}
Next assume that it is satisfied for some $k$ and let us show that it is also satisfied for $n+m\le k+1$. So it remains to show that it is satisfied for $n+m=k+1$, i.e.,  to show that
\begin{equation}
\begin{aligned}
Y^{n+1,k+1-n}\leq Y^{n,k+1-n}\leq Y^{n,k+2-n},\ \ \text{for any}\ \  n=0,1,2,...,k+1.
\end{aligned}
\end{equation}
Here, we will only show $Y^{n+1,k+1-n}\leq Y^{n,k+1-n}$ and the remainder is similar. By assumption, we have $Y^{n,k-n}\leq Y^{n-1,k-n}$ for $n=0,1,2,...,k+1$, where we use the notation \eqref{notation of n=-1 m=-1}   for $n=0$ and $n=k+1$. Since $f(t,y,y',z)$ is non-decreasing w.r.t $y'$ and $g(t,y,y')$ is non-decreasing w.r.t  $y,y'$, then for any
$n=0,1,...,k+1$, we have
\begin{equation}
\begin{aligned}
f(s,y,\E[Y_s^{n,k-n}],z)\leq & f(s,y,\E[Y_s^{n-1,k-n}],z)\ \ \ \text{and}\\
-(n+1)\left(y-g(s,Y^{n,k-n}_s,\E[Y^{n,k-n}_s])\right)^+ \leq & -n\left(y-g(s,Y^{n-1,k-n}_s,\E[Y^{n-1,k-n}_s])\right)^+. \\
\end{aligned}
\end{equation}
Taking now into account of the monotonicity property of $h$ (see (A2)-(ii),(a)) we deduce, by comparison, that $Y^{n+1,k-n}\leq Y^{n+1,k+1-n}$. Then the induction is completed.\par
Next for simplicity, we write
$$\forall n,m\ge 0 \mx{ and }s\le T,\,\,L^{n,m}_s:=h(s,Y_s^{n,m},\E[Y_s^{n,m}]),\ U^{n,m}_s:=g(s,Y_s^{n,m},\E[Y_s^{n,m}]).$$
Then we have the following estimates.
\begin{lemma}\label{estimate of YZKK nm of BSDE}
There exists a constant $C\ge 0$ such that:\\ i) For any $n,m \geq  0$,
\begin{equation}\lb{roughestimateynm}
\begin{aligned}
\sup_{0\leq t\leq T}\E[(Y^{n,m}_t)^2]+&\E\Big [\int_0^T |Z^{n,m}_t|^2 dt\Big ]\leq C.
\end{aligned}
\end{equation}
ii) For any $n,m\ge 0$,
\begin{equation}\lb{eq11bis}
\begin{aligned}
\E\Big[m^2\left(\int_0^T\left(Y_s^{n,m}-L^{n-1,m-1}_s\right)^- dt\right)^2+n^2\left(\int_0^T\left(Y_s^{n,m}-U^{n-1,m-1}_s\right)^+ dt\right)^2\Big ]\leq C.
\end{aligned}
\end{equation}
\end{lemma}
\begin{proof}Let $\eta$ and $\tau$ be stopping times such that $\eta\le \tau$ $\p-a.s.$. We define two sequences of stopping times $\{T_k\}_{k\ge 1}$ and $\{S_k\}_{k\ge 0}$ by
\begin{equation*}
\begin{aligned}
S_0=\eta,\ \  &T_k=\inf\{S_{k-1}\leq r\leq T: Y^{n,m}_r = U^{n-1,m-1}_s\}\wedge \tau
, \ \ k\ge 1,\\
&S_k=\inf\{T_k\leq r\leq T: Y^{n,m}_r = L^{n-1,m-1}_s\}\wedge \tau,\ \ k\ge 1. \end{aligned}
\end{equation*}
As $h<g$ then when $k$ tends to $+\infty$, we have $T_k\nearrow \t$ and $S_k \nearrow \t$. Since $Y^{n,m}\geq L^{n-1,m-1}$ on $[T_k,S_k]\cap \{T_k<S_k\}$, we have
\begin{equation}
\begin{aligned}
Y_{T_k}^{n,m}=Y_{S_k}^{n,m}  +\int_{T_k}^{S_k} f(s,Y_s^{n,m},\E[Y_s^{n-1,m-1}],Z_s^{n,m})ds-\int_{T_k}^{S_k} Z_s^{n,m} dB_s- n\int_{T_k}^{S_k} \left(Y_s^{n,m}-U^{n-1,m-1}_s\right)^+ ds,
\end{aligned}
\end{equation}
which implies, for all $k\geq 1$,
\begin{equation}
\begin{aligned}
n\int_{T_k}^{S_k} &\left(Y_s^{n,m}-U^{n-1,m-1}_s\right)^+ ds \leq  X_{S_k}-X_{T_k}+\int_{T_k}^{S_k} f(s,Y_s^{n,m},\E[Y_s^{n-1,m-1}],Z_s^{n,m})ds-\int_{T_k}^{S_k} Z_s^{n,m} dB_s\\
\leq & -\int_{T_k}^{S_k} J_s dB_s+V_{S_k}^+-V_{T_k}^+-V_{S_k}^-+V_{T_k}^--\int_{T_k}^{S_k} Z_s^{n,m} dB_s+\int_{T_k}^{S_k} |f(s,Y_s^{n,m},\E[Y_s^{n-1,m-1}],Z_s^{n,m})|ds\\
\leq & -\int_{T_k}^{S_k} J_s dB_s+V_{S_k}^+-V_{T_k}^+-\int_{T_k}^{S_k} Z_s^{n,m} dB_s+\int_{T_k}^{S_k} |f(s,Y_s^{n,m},\E[Y_s^{n-1,m-1}],Z_s^{n,m})|ds.
\end{aligned}
\end{equation}
Summing in $k$, we get
\begin{equation}\label{eq:1.13}
\begin{aligned}
n\int_{\h}^{\t} \left(Y_s^{n,m}-U^{n-1,m-1}_s\right)^+ ds \leq & -\int_{\h}^{\t} \left(J_s+Z_s^{n,m}\right)\sum_{k\ge 1}\{\mathbb{1}_{[T_k,S_k)}(s)\} dB_s+(V_\t^+-V_\h^+)\\&+\int_{\h}^{\t} |f(s,Y_s^{n,m},\E[Y_s^{n-1,m-1}],Z_s^{n,m})|ds.
\end{aligned}
\end{equation}
Take now $\h=t$ and $\t=T$, then squaring and taking expectation on the both sides of \eqref{eq:1.13}, we get
\begin{equation}\label{boundness of penalty U}
\begin{aligned}
n^2\E &\left(\int_{t}^{T} (Y_s^{n,m}-U^{n-1,m-1}_s)^+ ds \right)^2\\ &\leq  C\E\left(\int_{t}^{T} |J_s+Z_s^{n,m}|^2 ds+ \E[\left( V_T^+\right)^2]+\int_{t}^{T} |f(s,Y_s^{n,m},\E[Y_s^{n-1,m-1}],Z_s^{n,m})|^2ds\right)\\
&\leq C\E\left(1+\int_{t}^{T} |Z_s^{n,m}|^2 ds+\int_{t}^{T} (Y_s^{n,m})^2ds +\int_{t}^{T} (Y_s^{n-1,m-1})^2ds\right),
\end{aligned}
\end{equation}
where we have used the Lipschitz continuity of $f$ and (A2)-(i),(a). Similarly considering \eqref{MF-DRBSDE for nm} between $S_{k-1}$ and $T_k$, we obtain
\begin{equation}\label{eq:1.132}
\begin{aligned}
m\int_{\h}^{\t} \left(Y_s^{n,m}-L^{n-1,m-1}_s\right)^- ds \leq & \int_{\h}^{\t} \left(-J_s+Z_s^{n,m}\right)\sum_{k\ge 1}\{\mathbb{1}_{[S_{k-1},T_k)}(s)\} dB_s+(V_\t^--V_\h^-)\\&-\int_{\h}^{\t} f(s,Y_s^{n,m},\E[Y_s^{n-1,m-1}],Z_s^{n,m})ds,
\end{aligned}
\end{equation}
and then, as previously,  we have:
\begin{equation}\label{boundness of penalty L}
\begin{aligned}
m^2\E\left(\int_{t}^{T} \left(Y_s^{n,m}-L^{n-1,m-1}_s\right)^- ds\right)^2
\leq C\E\left(1+\int_{t}^{T} |Z_s^{n,m}|^2 ds+\int_{t}^{T} (Y_s^{n,m})^2 ds+\int_{t}^{T} (Y_s^{n-1,m-1})^2ds\right).
\end{aligned}
\end{equation}
Next, applying It\^{o}'s formula to $(Y^{n,m})^2$, for all constants $\alpha,\beta>0$, we have
\begin{equation}\label{Ito formula of Ynm2 of nm BSDE}
\begin{aligned}
\E&[(Y^{n,m}_t)^2]+\E\int_t^T |Z^{n,m}_s|^2 ds\\ =&\E[\xi^2]+2\E\left[\int_t^T Y_s^{n,m} f(s,Y_s^{n,m},\E[Y_s^{n-1,m-1}],Z_s^{n,m})ds\right]+\E\left[2m\int_t^T Y_s^{n,m}\left(Y_s^{n,m}-L^{n-1,m-1}_s\right)^- ds \right.\\
&\left.- 2n\int_t^T Y_s^{n,m}\left(Y_s^{n,m}-U^{n-1,m-1}_s\right)^+ ds\right]\\
 \leq &C \left(1+ (2+\frac{1}{\alpha})\E\int_t^T (Y_s^{n,m})^2 ds+\E\int_{t}^{T} (Y_s^{n-1,m-1})^2ds+\alpha\E\int_{t}^{T}|Z_s^{n,m}|^2 ds\right)+ \frac{1}{\beta}\E\sup_{s\in[t,T]}\left((L^{n-1,m-1}_s)^+\right)^2\\
&+\frac{1}{\beta}\E\sup_{s\in[t,T]}\left((U^{n-1,m-1}_s)^-\right)^2+\beta m^2\E\Big \{\left(\int_t^T(Y_s^{n,m}-L^{n-1,m-1}_s)^- ds\right)^2\Big\} +\beta n^2\E\Big \{\left(\int_t^T(Y_s^{n,m}-U^{n-1,m-1}_s)^+ ds\right)^2\Big\},
\end{aligned}
\end{equation}
where, according to \textbf{Assumption (A2)}-(ii),(d) we have \begin{equation}\label{estimates of supU and supL}
\begin{aligned}
\E\{\sup_{s\in[t,T]}\left((L^{n-1,m-1}_s)^+\right)^2+\sup_{s\in[t,T]}\left((U^{n-1,m-1}_s)^-\right)^2\} \leq 2\E\{\sup_{s\in[t,T]}|X_s|^2\}<+\infty.
\end{aligned}
\end{equation}
Then, substituting \eqref{boundness of penalty U}, \eqref{boundness of penalty L} and \eqref{estimates of supU and supL} to \eqref{Ito formula of Ynm2 of nm BSDE}, we get
\begin{equation}\label{Ito fomula of Y2}
\begin{aligned}
\E[(Y^{n,m}_t)^2]&+\E\int_t^T |Z^{n,m}_s|^2 ds \\
\leq &  C\E\left(1+(2+\frac{1}{\alpha})\int_t^T (Y_s^{n,m})^2 ds+\int_{t}^{T} (Y_s^{n-1,m-1})^2ds+ \alpha\int_{t}^{T}|Z_s^{n,m}|^2 ds\right)\\&+C\beta\E\left(1+\int_{t}^{T} (Y_s^{n,m})^2 ds+\int_{t}^{T} (Y_s^{n-1,m-1})^2ds+\int_{t}^{T} |Z_s^{n,m}|^2 ds\right).
\end{aligned}
\end{equation}
By choosing $\alpha$ and $\beta$ satisfying $(\alpha+\beta)C =\frac{1}{2} $ in the previous inequality, we obtain
\begin{equation}\label{Recursively using gronwall inequality}
\begin{aligned}
\E[(Y^{n,m}_t)^2]\leq  \widetilde{C}\E\left(1+\int_t^T (Y_s^{n,m})^2 ds+\int_{t}^{T} (Y_s^{n-1,m-1})^2ds\right),
\end{aligned}
\end{equation}
where  $\widetilde{C}$ is a constant  independent of $n$ and $m$.\par
Then, letting $n=m=0$ and recalling $Y^{-1,-1}=Y^{0,0}$, it follows that
\begin{equation*}\label{the inequality to use gronwall of Y00}
\begin{aligned}
\E[(Y^{0,0}_t)^2]\leq \widetilde{C}\E\left(1+2\int_t^T (Y_s^{0,0})^2 ds\right).
\end{aligned}
\end{equation*}
Using Gronwall's inequality, we have for all $t\in [0,T]$
\begin{equation}\label{ieq of Y00}
\begin{aligned}
\E[(Y^{0,0}_t)^2]\leq \gamma(t),
\end{aligned}
\end{equation}
where the function $\gamma: [0,T]\to R$ is defined by
\begin{equation*}\label{alpha_t defined for induction}
\begin{aligned}
\gamma(t):= \widetilde{C}e^{2\widetilde{C}(T-t)}.
\end{aligned}
\end{equation*}
Next, we assume that for $n=k$, the following inequality is guaranteed for all $t\in[0,T]$.
\begin{equation}\label{ieq of Yn0}
\begin{aligned}
\E[(Y^{n,0}_t)^2]\leq \gamma(t).
\end{aligned}
\end{equation}
Noticing the fact
$
\gamma(t)=\widetilde{C}(1+2\int_t^T \gamma(s)ds),
$
we have
\begin{equation*}
\begin{aligned}
\E[(Y^{k+1,0}_t)^2-\gamma(t)]\leq &\widetilde{C}\E\left(\int_t^T \left((Y_s^{k+1,0})^2-\gamma(s)\right)ds\right)+\widetilde{C}\E\left(\int_t^T \left((Y_s^{k,0})^2-\gamma(s)\right)ds\right)\\
\leq &\widetilde{C}\E\left(\int_t^T \left((Y_s^{k+1,0})^2-\gamma(s)\right)ds\right).
\end{aligned}
\end{equation*}
Applying Gronwall's inequality, we obtain  $$\E[(Y^{k+1,0}_t)^2]\leq\gamma(t)\ \ \ \text{for any }\  t\in[0,T],$$ which means \eqref{ieq of Yn0} is also guaranteed for $n=k+1$.
Thus, \eqref{ieq of Yn0} is correct for all $n\geq 0$.\par
Similarly, for a fixed $n$ and all $m\geq 0$, by using mathematical induction and Gronwall's inequality, we can obtain the following inequality for any $t\in[0,T]$,
\begin{equation}\label{ieq of Ynm}
\begin{aligned}
\E[(Y^{n,m}_t)^2]\leq \gamma(t).
\end{aligned}
\end{equation}
Since $n$ is arbitrary, \eqref{ieq of Ynm} is true for any $n,m\geq 0$.\\
Noticing $\gamma(t)$ is bounded, we have
\begin{equation}\label{S2 of Ynm}
\begin{aligned}
\sup_{0\leq t\leq T}\E[(Y^{n,m}_t)^2]\leq C, \ \text{uniformly in}\ \ (n,m).
\end{aligned}
\end{equation}
Then by \eqref{Ito fomula of Y2}, we can check
\begin{equation}\label{L2 of Znm}
\begin{aligned}
\E\int_0^T|Z^{n,m}_t|^2 dt\leq C, \ \text{uniformly in}\ \ (n,m).
\end{aligned}
\end{equation}
If we transfer \eqref{S2 of Ynm} and \eqref{L2 of Znm} into  \eqref{boundness of penalty U} and \eqref{boundness of penalty L} respectively, we get that
\begin{equation}
\begin{aligned}
n^2\E \left(\int_{t}^{T} (Y_s^{n,m}-U^{n-1,m-1}_s)^+ ds \right)^2\leq C \mx{ and }
m^2\E\left(\int_{t}^{T} \left(Y_s^{n,m}-L^{n-1,m-1}_s\right)^- ds\right)^2
\leq C.
\end{aligned}
\end{equation}
The result follows immediately.
\end{proof}\par
\begin{proposition}\label{Lemma: estimats of Yn Zn Kn and penalty term} For any $n\ge 0$, there exist processes $(Y^n,Z^n,K^n)$ that satisfy the following one barrier reflected BSDE:
\begin{equation}\label{MRBSDE of coeffcient n}\left\{
\begin{aligned}
&\mbox{ $Y^n\in \cS^2_c$, $K^{n,+}\in \cS^2_c$ non-decreasing ($K_0^{n,+}=0$) and $Z^n$ belongs to $\cH^{2,d}$}\,;\\
&Y_t^{n}= \xi +\int_t^T f(s,Y_s^{n},\E[Y_s^{n-1}],Z_s^{n})ds+ K_T^{n,+}-K_t^{n,+}-n\int_t^T(Y_s^{n}-U^{n-1}_s)^+ ds-\int_t^T Z_s^{n} dB_s,t\le T;\\
&Y^n_t\geq L^{n-1}_t,\ t\leq T,\mx{ and }\ \  \int_0^T(Y^n_t-L^{n-1}_t) dK^{n,+}_t=0,
\end{aligned}\right.
\end{equation}
where for any $n\ge 1$ and $t\le T$, $L^{n-1}_t=h(t,Y^n_t,\E[Y_t^{n-1}])$ and
 $U^{n-1}_t=g(t,Y^n_t,\E[Y_t^{n-1}])$. Moreover the following estimates hold true:
\begin{equation}\label{estimats of Yn Zn Kn and penalty term1}
\begin{aligned}
\E[\sup_{0\leq t\leq T}(Y^{n}_t)^2]+\E[\int_0^T |Z^{n}_t|^2 dt]+\E[(K_T^{n,+})^2] +\E
[ n^2\left(\int_0^T \left(Y_t^{n}-U^{n-1}_t\right)^+ dt\right)^2]\leq C,
\end{aligned}
\end{equation}
where $C$ is a constant which does not depend on $n$.
\end{proposition}
\begin{proof}: Recall that for any $n,m\ge 0$, $(Y^{n,m},Z^{n,m})$ verify \eqref{MF-DRBSDE for nm}, i.e., for any $t\le T$,
\begin{equation}\label{MF-DRBSDEnm}\left\{
\begin{aligned}
&Y^{n,m}\in\cS^2_c,\ \ Z^{n,m}\in \cH^{2,d},\\
&Y_t^{n,m}= \xi +\int_t^T f(s,Y_s^{n,m},\E[Y_s^{n-1,m-1}],Z_s^{n,m})ds-n\int_t^T \left(Y_s^{n,m}-g(s,Y_s^{n-1,m-1},\E[Y_s^{n-1,m-1}])\right)^+ ds
\\&\qquad \qquad +K^{n,m,+}_T-K^{n,m,+}_t-\int_t^T Z_s^{n,m} dB_s.\\
\end{aligned}\right.
\end{equation}
On the other hand, we know that for any $t\le T$ and for any $n,m\ge 0$, $Y^{n,m}_t\le Y^{n,m+1}_t$. Therefore for any $n\geq 0$ and $\tt$, let us set
\begin{equation}\label{eqyn}
Y^n_t=\lim_{m\rw \ft}Y^{n,m}_t.
\end{equation}
As $Y^{n,m}$ is continuous adapted, therefore $Y^n$ is a predictable process which is moreover l.s.c. On the other hand by Fatou's Lemma and \eqref{roughestimateynm} we deduce that:
\begin{equation}\label{roughestimateyn}
\begin{aligned}
\sup_{0\leq t\leq T}\E[(Y^{n}_t)^2]\le C
\end{aligned}
\end{equation}
Next the proof will be divided into three steps.
\ms

\nd \underline{Step 1}: There exists a constant $C$ such that for any $n\ge 0$,
\begin{equation}\label{estimats of Yn uniform}
\begin{aligned}
\E[\sup_{0\leq t\leq T}(Y^{n}_t)^2]\le C.
\end{aligned}
\end{equation}

For $m\ge 0$, let
$(\check Y^{m}_n,\check Z_n^{m})$ be the solution of the following BSDE:
\begin{equation}\label{MF-DRBSDEom}\left\{
\begin{aligned}
&\check Y_n^{m}\in\cS^2_c,\ \ \check Z_n^{m}\in \cH^{2,d};\\
&\check Y^{m}_n(t)= |\xi|+|X_T| +\int_t^T f(s,Y_s^{n},\E[Y_s^{n-1}],\check Z^{m}_n(s))ds+m\int_t^T (\check Y^{m}_n(s)-X_s)^-ds-\int_t^T \check Z^{m}_n(s)dB_s,\,\, t\le T.\\
\end{aligned}\right.
\end{equation}
Now let $(\check Y_n,\check Z_n,\check K_n)$ be the solution of the standard lower obstacle reflected BSDE associated with \\$(f(s,Y_s^{n},\E[Y_s^{n-1}],z),|\xi|+|X_T| ,(X_t)_{t\le T})$, i.e.,

\begin{equation}\label{MF-DRBSDEcheck}\left\{
\begin{aligned}
&\check Y_n\in\cS^2_c,\ \ \check Z_n \in \cH^{2,d}, \mx{ ($\check K_n\in\cS^2_c$, non-decreasing and $\check K_n(0)=0$)};\\
&\check Y_n(t)= |\xi|+|X_T| +\int_t^T f(s,Y_s^{n},\E[Y_s^{n-1}],\check Z_n(s))ds+\check K_n(T)-\check K_n(t)-\int_t^T \check Z_n(s)dB_s,\,\, t\le T.\\
&\check Y_n(t)\ge X_t ,\,\,\forall t\le T \mbox{ and }\int_0^T(\check Y_n(t)- X_t)d\check K_n(t)=0.
\end{aligned}\right.
\end{equation}Then there exists a constants $C\ge 0$ such that
\begin{equation}\label{estimats of Yn uniform2}
\begin{aligned}
\E[\sup_{0\leq t\leq T}(\check Y^{n}_t)^2]\le C \E[( |\xi|+|X_T|)^2+\sup_{s\le T}|X_t|^2+\int_0^T|f(s,Y_s^{n},\E[Y_s^{n-1}],0)|^2ds]\le C.
\end{aligned}
\end{equation}
The second inequality stems from the facts that $X\in \cS^2$, $\xi\in L^2(d\p)$ and finally in taking into account of \eqref{roughestimateyn} and since $f$ is Lipschitz and $(f(t,\omega,0,0,0))_{t\le T}\in \cH^{2,1}$. Then by standard comparison results we have:
\begin{equation}\label{etminyn2}
Y^{n,0}\le Y^{n,m-1}\le Y^{n,m}\le \check Y_n^m\le \check Y_n,
\end{equation}
since $h(t,y,y')\le X_t$ for any $t,y,y'$ (to infer the third inequality).

Next let $(\bar Y_n,\bar Z_n,\bar K_n)$ be the solution of the following upper barrier reflected BSDE:
\begin{equation}\label{MF-DRBSDEcheck2}\left\{
\begin{aligned}
&\bar Y_n\in\cS^2_c,\ \ \bar Z_n \in \cH^{2,d}, \mx{ ($\bar K_n\in\cS^2_c$ non-decreasing and $\check K_n(0)=0$)};\\
&\bar Y_n(t)= -|\xi|-|X_T| -\int_t^T f(s,\check Y_n(s),\E[\check Y_{n-1}(s)],\bar Z_n(s))ds-\bar K_n(T)+\bar K_n(t)-\int_t^T \bar Z_n(s)dB_s,\,\, t\le T;\\
&\bar Y_n(t)\le X_t ,\,\,\forall t\le T \mbox{ and }\int_0^T(\bar Y_n(t)- X_t)d\bar K_n(t)=0.
\end{aligned}\right.
\end{equation}
As previously in \eqref{estimats of Yn uniform2} and using the same inequality, there exists a constant $C\ge 0$ which does not depend on $n$ such that:
\begin{equation}\label{estimats of Yn uniform1}
\begin{aligned}
\E[\sup_{0\leq t\leq T}(\bar Y_{n}(t)^2]\le C \E[(|\xi|+|X_T|)^2+\sup_{t\le T}|X_t|^2+\int_0^T|f(s,\check Y_n(s),\E[\check Y_{n-1}(s)],0)|^2ds]\le C.
\end{aligned}
\end{equation}
On the other hand thanks to $g(t,y,y')\ge X_t$ and the monotonicity property of $f$ w.r.t $y$ and $y'$, by using a standard comparison argument we deduce that
$\p-a.s.$, for any $t\le T$, $\bar Y_{n}(t)\le Y^{n,0}_t$. Going back now to \eqref{etminyn2}, in using
the estimates obtained in \eqref{estimats of Yn uniform2} and \eqref{estimats of Yn uniform1}, one deduces the existence of a constant $C$ which does not depend on $n$ such that the estimate \eqref{estimats of Yn uniform} is satisfied. Finally let us notice that we have also
\begin{equation}\lb{estapproxk-}
\E[ n^2\left(\int_0^T \left(Y_t^{n}-U^{n-1}_t\right)^+ dt\right)^2]\leq C,
\end{equation}
which is a consequence of the convergence of $(Y^{n,m})_{m\ge 0}$ to $Y^n$ and the estimate \eqref{eq11bis}.
\bs

\nd \underline{Step 2}: For any $n\ge 0$,  $Y^n$ is continuous.
\ms

\nd Recall that $(Y^{n,m},Z^{n,m})$ verify \eqref{MF-DRBSDEnm} and let $n$ be fixed. For any $m\ge 0$, by
\eqref{eq11bis}
\begin{equation}\lb{estproxkmn}
\E[\{\underbrace{m\int_0^T(Y^{n,m}_s-L_s^{n-1,m-1})^-ds}_{K^{n,m,+}_T}\}^2]\le C,
\end{equation}
where $C$ is a constant independent of $m$ (and also $n$). Therefore using a result by A.Uppman (\cite{AreUppman}, Théorème 1, pp. 289), there exist a subsequence $(m_l)_{l\ge 0}$ and an adapted non-decreasing process $K^{n,+}$ such that for any stopping time $\t$,
$$
K^{n,m_l,+}_\t\rightharpoonup K^{n,+}_\t \mx{ weakly in the sense of $\sigma(L^1,L^\infty)$-topology, as }l\rw \ft.
$$But the uniform estimate \eqref{estproxkmn} implies that
$$
K^{n,m_l,+}_\t\rightharpoonup K^{n,+}_\t \mx{ weakly in
$L^2(d\p)$ as }l\rw \ft.$$
Next using estimate \eqref{roughestimateynm} and since $g$ is of linear growth at most, one deduces the existence of a constant $C_n$ which may depend on $n$ such that$$
\E\Big \{\int_0^T\{| f(s,Y_s^{n,m},\E[Y_s^{n-1,m-1}],Z_s^{n,m})|^2+|n\left(Y_s^{n,m}-g(s,Y_s^{n-1,m-1},\E[Y_s^{n-1,m-1}])\right)^+|^2+|Z_s^{n,m}|^2\}ds\Big \}\le C_n.
$$
Therefore there exists a subsequence which we still denote by $(m_l)_{l\ge 0}$ such that
$$
(f(s,Y_s^{n,m_l},\E[Y_s^{n-1,m_l-1}],Z_s^{n,m_l}))_{s\le T}\rightharpoonup (\Phi_n(s))_{s\le T}\mx{ weakly in $\cH^{2,1}$ as $l\rw \ft$,}
$$
$$
(n(Y_s^{n,m_l}-g(s,Y_s^{n-1,m_l-1},\E[Y_s^{n-1,m_l-1}])^+)_{s\le T}\rightharpoonup (\Sigma_n(s))_{s\le T}\mx{ weakly in $\cH^{2,1}$ as $l\rw \ft$}
$$
and
$$
(Z_s^{n,m_l})_{s\le T}\rightharpoonup (Z_n(s))_{s\le T}\mx{ weakly in $\cH^{2,d}$ as $l\rw \ft$}.
$$
Going back now to \eqref{MRBSDE of coeffcient n} and considering the equation between $0$ and $\t$ with $m_l$ and finally sending $l$ to infinite, we deduce that for any stopping time $\t$ it holds:$$\p-a.s.,\,\,
Y_\t^{n}= Y^{n}_0-\int_0^\t \Phi_n(s)ds+
\int_0^\t \Sigma_n(s)ds-K_\t^{n,+}+\int_0^\t Z_n(s)dB_s.$$
Next let us consider $(\t_k)_{k\ge 0}$ a decreasing sequence of stopping times which converges to some $\theta$. As the filtration $(\cF_t)_{t\le T}$ satisfies the usual conditions (and then it is right continuous) then $\theta$ is also an $(\cF_t)_{t\le T}$-stopping time valued in $[0,T]$. On the other hand for any $k\ge 0$,
$$\p-a.s.,\,\,
Y_{\t_k}^{n}-Y_\theta^{n}= -\int_\theta^{\t_k} \Phi_n(s)ds+
\int_\theta^{\t_k}  \Sigma_n(s)ds-(K_{\t_k}^{n,+}-K_\theta^{n,+})+\int_\theta^{\t_k} Z_n(s)dB_s$$
which, in taking expectation, yields:
$$\E[
Y_{\t_k}^{n}-Y_\theta^{n}]= \E[-\int_\theta^{\t_k} \Phi_n(s)ds+
\int_\theta^{\t_k}  \Sigma_n(s)ds]-\E[(K_{\t_k}^{n,+}-K_\theta^{n,+})].$$
But the first term in the right-hand side converges to 0 as $k\rw \infty$. On the other hand by the weak limit and \eqref{eq:1.132} we have:
\begin{equation}
\begin{aligned}
\E[(K_{\t_k}^{n,+}-K_\theta^{n,+})]&=\lim_{l\rw \ft}\E[(K_{\t_k}^{n,m_l,+}-K_\theta^{n,m_l,+})]\le \limsup_{l\rw \ft}\E[(V_{\t_k}^+-V_\theta^+)+\int_{\theta}^{\t_k} |f(s,Y_s^{n,m_l},\E[Y_s^{n-1,m_l-1}],Z_s^{n,m_l})|ds]\\&\le
\E[(V_{\t_k}^+-V_\theta^+)]+\limsup_{l\rw \ft}\E[\int_{\theta}^{\t_k} |f(s,Y_s^{n,m_l},\E[Y_s^{n-1,m_l-1}],Z_s^{n,m_l})|ds].
\end{aligned}
\end{equation}
As the process
$( |f(s,Y_s^{n,m_l},\E[Y_s^{n-1,m_l-1}],Z_s^{n,m_l})|)_{s\le T}$ belongs uniformly to $\cH^{2,1}$, then \\$\lim_{k\rw \ft}\E[(K_{\t_k}^{n,+}-K_\theta^{n,+})]=0$ and thus $\lim_{k\rw \ft}\E[
Y_{\t_k}^{n}-Y_\theta^{n}]=0$. Now as $Y$ is a predictable (and then optional) process, $\tau$  and $\t_k$ are arbitrary, then, by a result by Dellacherie-Meyer (\cite{dlm}, pp.120, Theor\`eme 48) the process $Y^n$ is right continuous. In the same way if $(\t_k)_{k\ge 0}$ is an increasing sequence of stopping times which converges to a predictable stopping time then
$\E[
Y_{\t_k}^{n}-Y_\theta^{n}]=0$. As $Y^n$ is predictable then $Y^n$ is left continuous. Consequently $Y^n$ is continuous (see \cite{dlm}, pp.120, Th\'eor\`eme 48).
\ms

\nd \underline{Step 3}: There exist processes $(Y^n,Z^n,K^n)$ that satisfy \eqref{MRBSDE of coeffcient n}-\eqref{estimats of Yn Zn Kn and penalty term1}.

Actually as $Y^{n,m}\nearrow Y^n$ and those processes are continuous then thanks to Dini's theorem the convergence is uniform $\omega$ by $\omega$ $\p-a.s.$. Finally by the Lebesque dominated converge theorem and \eqref{estimats of Yn uniform} we have
$$
\lim_{m\rw \infty}\E[\sup_{t\le T}|Y^{n,m}_t-Y^n_t|^2]=0.
$$
Next going back to \eqref{MF-DRBSDEnm} and using It\^o's formula  with $(Y^{n,m}-Y^{n,q})^2$ and taking into account of \eqref{eq11bis}, we obtain:
$$
\E[\int_0^T|Z_s^{n,m}-Z^{n,q}_s|^2ds]\rw 0 \mx{ as $m,q\rw \ft$}.
$$
Thus let us denote by $Z^n$ the limit in $\cH^{2,d}$ of the sequence $(Z^{n,m})_{m\ge 0}$ which exists since it is of Cauchy type in this normed complete linear space. Then we have also the convergence of
$( f(s,Y_s^{n,m},\E[Y_s^{n-1,m-1}],Z_s^{n,m}))_{s\le T}$ and $n(Y_s^{n,m}-g(s,Y_s^{n-1,m-1},\E[Y_s^{n-1,m-1}]))^+)_{s\le T}$ in $\cH^{2,1}$ toward
$( f(s,Y_s^{n},\E[Y_s^{n-1}],Z_s^{n}))_{s\le T}$ and\\ $n(Y_s^{n}-g(s,Y_s^{n-1},\E[Y_s^{n-1}]))^+)_{s\le T}$ respectively. Next from \eqref{estimate of YZKK nm of BSDE}, in taking the limit w.r.t $m$, one deduces that
$$
\E[\int_0^T(Y^n_s-h(s,Y^n_s,\E[Y^{n-1}_s]))^-ds]=0
$$
which implies, by continuity, that for any $s\le T$, $Y_s\ge h(s,Y^n_s,\E[Y^{n-1}_s])$.
Finally for any $t\le T$, let us set
$$
K^{n,+}_t=Y^n_0-Y^n_t-\int_0^tf(s,Y_s^{n},\E[Y_s^{n-1}],Z_s^{n})ds+n\int_0^t(Y_s^{n}-U^{n-1}_s)^+ ds
+\int_0^t Z_s^{n} dB_s.
$$
Therefore $K^{n,+}$ is continuous and is nothing but the limit w.r.t $m$ in $\cS^2_c$ of $K^{n,m,+}$. As $K^{n,m,+}$ is non-decreasing then $K^{n,+}$ is also non-decreasing.
Finally let us notice that for any $m\ge 0$, we have
\begin{equation}
\begin{aligned}\forall s\le T,
Y^{n,m}_s\wedge L^{n-1,m-1}_s\leq Y^{n,m}_s
\end{aligned}
\end{equation}
and\begin{equation}
\begin{aligned}
\int_0^T (Y^{n,m}_s-Y^{n,m}_s\wedge L^{n-1,m-1}_s)d K^{n,m,+}_s=0.
\end{aligned}
\end{equation}
As $( (Y^{n,m}_s-Y^{n,m}_s\wedge L^{n-1,m-1}_s)_{s\le T})_{m\ge 0}$ and $ (K^{n,m,+})_{m\ge 0}$ converge in $\cS^2_c$ toward $ (Y^{n}_s-L^{n-1}_s)_{s\le T}$ and $K^{n,+}$ respectively (as $m\rwft$) then by Helly's Theorem (\cite{helly}, pp.370) we have:
$$
\int_0^T(Y^n_s-h(s,Y^n_s,\E[Y^{n-1}_s]))dK^{n,+}_s=0.
$$
It follows that $(Y^n,Z^n,K^{n,+})$ verify the reflected BSDE \eqref{MRBSDE of coeffcient n}. The remaining estimates of \eqref{estimats of Yn Zn Kn and penalty term1} stem from: (i) the convergence of the sequence $(Z^{n,m})_{m\ge 0}$ in $\cH^{2,d}$ and \eqref{L2 of Znm}; (ii) the above definition of $K^{n,+}$ in combination with \eqref{estapproxk-}  mainly.
\end{proof}
\par
To resume,  from \eqref{inequality of Ynm Yn+1m and Ynm+1}, we deduce that for any $n\ge 0$, $Y^{n+1}\le Y^n$. Then let us set
$$
\p-a.s.,\,\,\forall t\le T,\,\,Y_t=\lim_{n\rwft}Y^n_t.
$$
Note that by \eqref{estimats of Yn uniform}, for any stopping time $\t$, $(Y^n_\t)_{n\ge 0}$ converges toward $Y_\t$ in $L^1(d\p)$ . We are now ready to give the main result of this part.
\begin{theorem} There exist processes $Z$ and $K^\pm$ such that $(Y,Z,K^\pm)$ verify:
\begin{equation}\label{MF-DRBSDE2}\left\{
\begin{aligned}
&Y\in \mathcal{S}^2_c, Z\in \mathcal{H}^{2,d}, K^\pm\in \mathcal{S}^2_{ci};\\
&Y_t= \xi +\int_t^T f(s,Y_s,\E[Y_s],Z_s)ds+ K_T^+-K_t^+-(K_T^- -K_t^-)-\int_t^T Z_s dB_s,\,\,t\le T;\\
&h(t,Y_t,E[Y_t])\leq Y_t\leq g(t,Y_t,E[Y_t]),\ \  t\leq T;\\
&\int_0^T\left(Y_t-h(t,Y_t,E[Y_t])\right)dK_t^+=\int_0^T\left(g(t,Y_t,E[Y_t])-Y_t\right)dK_t^-=0.
\end{aligned}\right.
\end{equation}
\end{theorem}
\begin{proof}: For any $n\ge 0$, $(Y^n,Z^n,K^{n,+})$ verify \eqref{MRBSDE of coeffcient n}, i.e.,
\begin{equation}\label{MRBSDE of coeffcient n1}\left\{
\begin{aligned}
&Y_t^{n}= \xi +\int_t^T f(s,Y_s^{n},\E[Y_s^{n-1}],Z_s^{n})ds+ K_T^{n,+}-K_t^{n,+}-n\int_t^T(Y_s^{n}-U^{n-1}_s)^+ ds-\int_t^T Z_s^{n} dB_s,\,\,t\le T;\\
&Y^n_t\geq L^{n-1}_t,\ t\leq T, \  \ \text{and}\ \  \int_0^T(Y^n_t-L^{n-1}_t) dK^{n,+}_t=0.
\end{aligned}\right.
\end{equation}
First note that by \eqref{estimats of Yn Zn Kn and penalty term1} we have
\begin{equation}\label{eq147}
\begin{aligned}
\E\left[(K_T^{n,+})^2\right] +\E\left[ n^2\left(\int_0^T \left(Y_t^{n}-U^{n-1}_t\right)^+ dt\right)^2\right]\leq C
\end{aligned}
\end{equation}
where $C$ is constant independent of $n$. On the other hand by \eqref{eq:1.13} and \eqref{eq:1.132} for any stopping times $\h \le \t$ we have:
\begin{equation}\lb{eq153}
\begin{aligned}
\E[n\int_\h^\t \left(Y_s^{n}-g(s,Y_s^{n-1},\E[Y_s^{n-1 }])\right)^+ ds]
\leq \E[(V_\t^+-V_\h^+)]+\E[\int_{\h}^{\t} |f(s,Y_s^{n},\E[Y_s^{n-1}],Z_s^{n})|ds].
\end{aligned}
\end{equation}
 and
\begin{equation}\lb{eq154}
\E[K^{n,+}_\t-K^{n,+}_\h]\le
\E[(V_\t^--V_\h^-)]+\E[\int_{\h}^{\t} |f(s,Y_s^{n},\E[Y_s^{n-1}],Z_s^{n})|ds].
\end{equation}
Therefore once more by Uppman's result  (\cite{AreUppman}, Théorème 1, pp. 289) and \eqref{eq147}, there exists a subsequence $(n_l)_{l\ge 0}$ such that for any stopping time $\t$,
$$
K^{n_l,+}_\t\rightharpoonup K^{+}_\t \mx{ weakly in the sense of $\sigma(L^1,L^\infty)$-topology as }l\rw \ft
$$
and
$$
K^{n_l,-}_\t:=n\int_0^\t \left(Y_s^{n}-g(s,Y_s^{n-1},\E[Y_s^{n-1 }])\right)^+ ds\,\,\rightharpoonup K^{-}_\t \mx{ weakly in the sense of $\sigma(L^1,L^\infty)$-topology as }l\rw \ft ,
$$
where $K^\pm$ are adapted non-decreasing processes. Next using estimate \eqref{eq147} one deduces that the previous convergences hold also in $L^2(d\p)$ weakly and not only for the  $\sigma(L^1,L^\infty)$-topology. Next since by \eqref{estimats of Yn Zn Kn and penalty term1}, there exists a constant $C$ independent of $n_l$ such that
$$
\E[\int_0^T\{|f(s,Y^{n_l}_s,\E[Y^{n_l-1}_s],Z^{n_l}_s)|^2+|Z^{n_l}_s)|^2\}ds]\le C,
$$
then there exists a subsequence which we still denote by $(n_l)_{l\ge 0}$ such that
$((f(s,Y^{n_l}_s,\E[Y^{n_l-1}_s],Z^{n_l}_s))_{s\le T})_{l\ge 0}$ (resp. $((|f(s,Y^{n_l}_s,\E[Y^{n_l-1}_s],Z^{n_l}_s)|)_{s\le T})_{l\ge 0}$; resp. $(Z^{n_l})_{l\ge 0}$) converges weakly in $\cH^{2,1}$ (resp. $\cH^{2,1}$; resp. $\cH^{2,d}$) to some process
$(\Phi(s))_{s\le T}$ (resp. $\Xi$; resp. $Z$) as $l\rwft$. Therefore from \eqref{eq153} and \eqref{eq154} and by Fatou's Lemma we deduce that : For any stopping times $\h \le \t$,
\begin{equation}\lb{eqk+}
\E[K^{+}_\t-K^{  +}_\h]\le
\E[(V_\t^--V_\h^-)]+\E[\int_{\h}^{\t} \Xi (s)ds  ]
\end{equation}
and
\begin{equation}\lb{eqk-}
\E[K^{-}_\t-K^{ - }_\h]\le
\E[(V_\t^+-V_\h^+)]+\E[\int_{\h}^{\t} \Xi(s)ds  ].
\end{equation}
Next from the equation \eqref{MRBSDE of coeffcient n1} written forwardly, we deduce that for any stopping time $\t$ it holds:
\begin{equation}\lb{eqyt}
Y_\t=Y_0-\int_0^\t\Phi(s)ds-K^+_\t+K^-_\t+\int_0^\t Z_sdB_s.
\end{equation}
Now as the process $Y$ is predictable then if $(\t_k)_{k\ge 0}$ is a decreasing sequence of stopping times that converges to $\th$ then
$$
\lim_{n\rwft}\E[Y_{\t_k}-Y_\th]=0
$$
since $K^+$ and $K^-$ satisfy the inequalities in \eqref{eqk+} and \eqref{eqk-} respectively. Thus $Y$ is right continuous (\cite{dlm}, pp.120, Theor\`eme 48).

In the same way if $(\t_k)_{k\ge 0}$ is an increasing sequence of predictable stopping times that converges to  a predictable stopping time $\th$ then  $$
\lim_{n\rwft}\E[Y_{\t_k}-Y_\th]=0.
$$
Therefore, similarly, the predictable process $Y$ is left continuous and then continuous.

Now we resume as we did in the proof of Proposition \ref{Lemma: estimats of Yn Zn Kn and penalty term}. Thanks to Dini's Theorem and dominated convergence Theorem, the convergence of the sequence $(Y^n)_{n\ge 0}$ to $Y$ holds in $\cS^2_c$ and by It\^o's formula applied with $(Y^n-Y^q)^2$ we obtain that the sequence of processes $(Z^n)_{n\ge 0}$ is of Cauchy type in $\cH^{2,d}$ and then converges to a process $Z$ which belongs to
$\cH^{2,d}$. Now from the inequality \eqref{estimats of Yn Zn Kn and penalty term1} we deduce that
$$
\E[\int_0^T(Y^n_s-g(s,Y^{n-1}_s,\E[Y^{n-1}_s]))^+ds\leq Cn^{-1}.
$$
Sending $n$ to $+\infty$ and using the uniform convergence of $(Y^n)_n$ and continuity of $Y$ and
$(g(t,Y_t,\E[Y_t]))_{t\le T}$ to deduce that $Y_t\le g(t,Y_t,\E[Y_t])$ for any $t\le T$. Thus for any $t\leq T$,
$$
h(t,Y_t,\E[Y_t])\le Y_t\le g(t,Y_t,\E[Y_t])
$$
since $Y^n\ge L^{n-1}$. Next the following properties hold true:
\begin{equation}
\begin{aligned}\forall s\le T,
Y^{n}_s\vee U^{n-1}_s\geq Y^{n}_s \mx{ and }\int_0^T (Y^{n}_s\vee U^{n-1}_s-Y^{n}_s )d K^{n,-}_s=0.
\end{aligned}
\end{equation}
Combining this with the backward equation \eqref{MRBSDE of coeffcient n1} that $Y^n$ verifies and based on already known results on double barrier reflected BSDEs and zero-sum Dynkin games (see \cites{ck96,hl00} for more details), one deduces that for any $t\le T$,
\begin{equation}\label{Dykin Game of ynm}
\begin{aligned}
Y^{n}_t=& \esssup_{\tau\geq t}\essinf_{\sigma \ge t}\E\left[ \int_t^{\tau\wedge \s} f(s,Y^{n  }_s,\E[Y^{n-1    }_s],Z^{n  }_s)ds+\right.\\ &\left.\left.
\qquad \qquad Y^{n}_\s\vee g(s,Y^{n-1}_s,\E[Y^{n-1}_s])_{|s=\s}\mathbb{1}_{\{\s<\t\}} + (L^{n-1}_s)_{|s=\t}\mathbb{1}_{\{\tau\le \s, \t<T\}}
+\xi\mathbb{1}_{\{\tau=\sigma=T\}}\right|\mathcal{F}_t\right].
\end{aligned}
\end{equation}
So let now $\bar Y$ be the process defined as follows: For any $t\le T$,
\begin{equation}\label{Dykin Game of y}
\begin{aligned}
\bar Y_t=& \esssup_{\tau\geq t}\essinf_{\sigma \ge t}\E\left[ \int_t^{\tau\wedge \s} f(s,Y_s,\E[Y_s],Z_s)ds+\right.\\ &\left.\left.
\qquad \qquad g(s,Y_s,\E[Y_s])_{|s=\s}\mathbb{1}_{\{\s<\t\}} + h(s,Y_s,\E[Y_s])_{|s=\t}\mathbb{1}_{\{\tau\le \s, \t<T\}}
+\xi\mathbb{1}_{\{\tau=\sigma=T\}}\right|\mathcal{F}_t\right].
\end{aligned}
\end{equation}
The process $\bar Y$ is related to double barrier reflected BSDEs in the following way: There exist processes
$\bar Z\in \cH^{2,d}$ and non-decreasing continuous processes $K^\pm\in \cS^2$ $(K^\pm=0)$ such that
\begin{equation}\label{MF-DRBSDE3}\left\{
\begin{aligned}
&\bar Y_t= \xi +\int_t^T f(s,Y_s,\E[Y_s],Z_s)ds+ \bar K_T^+-\bar K_t^+-(\bar K_T^- -\bar K_t^-)-\int_t^T \bar Z_s dB_s,\,t\le T;\\
&h(t,Y_t,\E[Y_t])\leq \bar Y_t\leq g(t,Y_t,\E[Y_t]),\  t\leq T;\\
&\int_0^T\left(\bar Y_t-h(t,Y_t,\E[Y_t])\right)d\bar K_t^+=\int_0^T\left(g(t,Y_t,\E[Y_t])-\bar Y_t\right)d\bar K_t^-=0.
\end{aligned}\right.
\end{equation}
This is mainly due to the facts that Mokobodski's condition is satisfied, i.e, for any $t\le T$, $h(t,Y_t,\E[Y_t])\le X_t\le g(t,Y_t,\E[Y_t])$ and all those processes belong to $\cS^2_c$ (one can see e.g. \cites{bhm, ck96,hl00}). Next since $$|\underset{\tau \ge
t}{\esssup}\,\underset{\s \ge
t}{\essinf}\,\Sigma^1_{t,\tau,\s}-\underset{\tau \ge
t}{\esssup}\,\underset{\s \ge
t}{\essinf}\,\Sigma^2_{t,\tau,\s}|\le \underset{\tau \ge
t}{\esssup}\,\underset{\s \ge
t}{\esssup}\,|\Sigma^1_{t,\tau,\s}-\Sigma^2_{t,\tau,\s}|,$$ then for any $t\le T$,

\begin{align}
|Y^n_t&-\bar Y_t|\le  \E\{\int_0^T|f(s,Y_s,\E[Y_s],Z_s)-f(s,Y^{n}_s,\E[Y^{n-1}_s],Z^{n}_s)|ds+\nn\\&
\sup_{s\le T}|
Y^{n}_s\vee g(s,Y^{n-1}_s,\E[Y^{n-1}_s])-g(s,Y_s,\E[Y_s])|+\sup_{s\le T}|h(s,Y_s,\E[Y_s])-h(s,Y^{n-1}_s,\E[Y^{n-1}_s])||\mathcal{F}_t\}.\nn
\end{align}
Then by Doob's inequality (\cite{revuzyor}, pp.54), there exists a constant $C\ge 0$ such that:
\begin{align}
&\E[\sup_{t\le T}|Y^n_t-\bar Y_t|^2]\nn\\&\le C\Big \{\E\{\sup_{s\le T}|
Y^{n}_s\vee g(s,Y^{n-1}_s,\E[Y^{n-1}_s])-g(s,Y_s,\E[Y_s])|^2\}+\E\{\sup_{s\le T}|h(s,Y_s,\E[Y_s])-h(s,Y^{n-1}_s,\E[Y^{n-1}_s])|^2\}
\nn \\&
\qquad \qquad \qquad +\E\{\int_0^T|f(s,Y_s,\E[Y_s],Z_s)-f(s,Y^{n  }_s,\E[Y^{n-1}_s],Z^{n  }_s)|^2ds\}\Big \}=\ca_1^n+\ca_2^n+\ca_3^n \nn.
\end{align}
But $\lim_{n\rwft}\ca_3^n=0$ since $(Y^n)_n$ converges  to $Y$ in $\cS^2_c$ and $(Z^n)_n$ converges to $Z$ in $\cH^{2,d}$ and $f$ is uniformly Lipschitz w.r.t $(y,y',z)$. Next, $\lim_{n\rwft}\ca_2^n=0$ by Dini's Theorem and Lebesgue dominated convergence one since $h$ is monotone and $h(t,Y_t,\E[Y_t])$ is continuous w.r.t $t$. Finally using the same arguments one deduces that $\lim_{n\rwft}\ca_1^n=0$ as for any $t\le T$, $Y_t\le g(t,Y_t,\E[Y_t])$. Thus
$$
\lim_{n\rwft}\E[\sup_{t\le T}|Y^n_t-\bar Y_t|^2]=0
$$
and then $Y=\bar Y$. Next by \eqref{eqyt} and \eqref{MF-DRBSDE3} one deduces that $Z=\bar Z$. Therefore $(Y,Z,\bar K^\pm)$ verify \eqref{MF-DRBSDE2}, i.e., it is a solution of the mean-field reflected BSDE associated with $(f,\xi,h,g)$.
\end{proof}
As a by-product we obtain the following result related to the approximation of the solution of \eqref{MF-DRBSDE}. Its proof is based on uniqueness of the solution of that equation. \begin{corollary}Assume that the mean-field RBSDE \eqref{MF-DRBSDE} has a unique solution $(Y,Z,K^\pm)$. Then
$$
\lim_{n\rwft}\E[\sup_{t\le T}|Y^n_t-Y_t|^2]=0.
$$
where for any $n\ge 0$, $Y^n$ is the limit in $\cS_c^2$ w.r.t $m$ of $Y^{n,m}$ solution of \eqref{MF-DRBSDE for nm} and which satisfies also \eqref{MRBSDE of coeffcient n}.
\end{corollary}
\begin{remark}The solution of \eqref{MF-DRBSDE for nm} is unique when:

(i) $g$ and $h$ verify moreover \eqref{lipgh} and \eqref{fix_uniq1} with $p=2$.

(ii) $f$ does not depend on $z$ or
$f(t,\omega,y,y',z)=\Phi(t,\omega,y,y')+\alpha_tz$ where $(\alpha_t)_{t\le T}$ is, e.g., an adapted bounded process.
\ms

\noindent In this latter case of $f$, uniqueness is obtained after a change of probability by using Girsanov's theorem and arguing mainly as in Section 3.1.
\end{remark}

\end{document}